\documentclass[11pt]{amsart}
\usepackage{amssymb,latexsym,amsmath}
\usepackage{bm}
\usepackage{upgreek}
\usepackage{hyperref}

\input{xypic}
\xyoption{all}

\usepackage{bm}
\usepackage{upgreek}
\makeatletter
\newcommand{\bfgreek}[1]{\bm{\@nameuse{up#1}}}
\makeatother

\begin{document}

\title[Hecke $L$-functions]{Notes on the arithmetic of Hecke $L$-functions}

\author{\bf A. Raghuram}
\address{A. Raghuram: Department of Mathematics\\ Fordham University Lincoln Center,  113 West 60th Street,  New York,  NY 10023, USA} 
\email{araghuram@fordham.edu}
\date{\today}

\begin{abstract}
This is an expository article that concerns the various related notions of algebraic id\`ele class characters, the {\it Gr\"o\ss encharaktere} of Hecke, 
and cohomological automorphic representations of ${\rm GL}(1)$, all under the general title of algebraic Hecke characters. 
The first part of the article systematically lays the foundations of algebraic Hecke characters. The only pre-requisites are: basic algebraic number theory, familiarity with the adelic language, and basic sheaf theory. Observations that play a crucial role in the arithmetic of automorphic $L$-functions are also discussed. 
The second part of the article, on the ratios of successive critical values of the Hecke $L$-function attached to an algebraic Hecke character, concerns certain variations on a theorem of Harder \cite{harder}, especially drawing attention to a delicate signature that apparently has not been noticed before. 
\end{abstract}

\maketitle


\def\g{\mathfrak{g}}
\def\m{\mathfrak{m}}
\def\k{\mathfrak{k}}
\def\z{\mathfrak{z}}
\def\h{{\mathfrak h}}
\def\gl{\mathfrak{gl}}
\def\p{\mathfrak{p}}
\def\q{\mathfrak{q}}
\def\f{\mathfrak{f}}
\def\a{\mathfrak{a}}
\def\c{\mathfrak{c}}
\def\l{\mathfrak{l}}
\def\i{\mathfrak{i}}
\def\ul{\underline }

\def\Ext{{\rm Ext}}
\def\Hom{{\rm Hom}}
\def\Ind{{\rm Ind}}
\def\aInd{{}^{\rm a}{\rm Ind}}
\def\cm{{\rm cm}}

\def\GL{{\rm GL}}
\def\SO{{\rm SO}}

\def\O{\mathcal{O}}

\def\R{\mathbb{R}}
\def\C{\mathbb{C}}
\def\Z{\mathbb{Z}}
\def\Q{\mathbb{Q}}
\def\A{\mathbb{A}}
\def\I{\mathbb{I}}
\def\J{\mathbb{J}}
\def\U{\mathbb{U}}
\def\P{\mathbb{P}}
\def\bL{\mathbb{L}}

\def\G{\mathcal{G}}
\def\Cl{\mathcal{C}l}
\def\cN{\mathcal{N}}
\def\cG{\mathcal{G}}
\def\cM{\mathcal{M}}
\def\tM{\widetilde{\mathcal{M}}}

\def\w{\wedge}

\def\HC{{\rm HC}}
\def\proj{{\rm proj}}

\def\cH{\mathcal{H}}

\def\to{\rightarrow}
\def\To{\longrightarrow}

\def\1{1\!\!1}
\def\dim{{\rm dim}}

\def\th{^{\rm th}}
\def\isom{\approx}

\def\CE{\mathcal{C}\mathcal{E}}

\def\Crit{{\rm Crit}}
\def\autc{{\rm Aut}(\C)}
\def\log{{\rm log}}
\def\GalQ{{\rm Gal}(\bar\Q/\Q)}
\def\Gal{{\rm Gal}}
\def\Res{{\rm Res}}


\def\bpi{\bfgreek{pi}}


\setcounter{tocdepth}{2}

\newtheorem{thm}[equation]{Theorem}
\newtheorem{cor}[equation]{Corollary}
\newtheorem{lemma}[equation]{Lemma}
\newtheorem{prop}[equation]{Proposition}
\newtheorem{con}[equation]{Conjecture}
\newtheorem{ass}[equation]{Assumption}
\newtheorem{defn}[equation]{Definition}
\newtheorem{rem}[equation]{Remark}
\newtheorem{exer}[equation]{Exercise}
\newtheorem{exam}[equation]{Example}


{\tiny \tableofcontents}

\bigskip
\section{\bf Introduction}

In \cite{weil}
Andr\'e Weil introduced a special class of characters on the group of id\`ele-classes of an algebraic number field $F$ which 
he called characters of type ($A_0$). The importance of these characters are best captured in Weil's words: 
\begin{verse}
`` {\it As to the non-trivial characters of type ($A_0$), some of them arise with the theory of abelian varieties with complex multiplication ... 
Taniyama has proved that the $L$-series attached to the characters of type ($A_0$) belonging to abelian varieties with complex multiplication are 
precisely those which occur in the zeta-functions of such varieties ... all that can be said here is that they tend to emphasize the importance of the 
characters we have discussed and of their remarkable properties.}" 
\end{verse} 
With the hindsight afforded by decades of development, we can now say that the modern arithmetic theory of automorphic forms has 
Weil's landmark 1955 paper as one of its points of origin.

\medskip

In this article, we are primarily concerned with the arithmetic aspects of automorphic forms and $L$-functions on ${\rm GL}(1).$ While working on the special values of automorphic $L$-functions, especially in the context of \cite{raghuram-CM} and \cite{raghuram-tot-imag}, the author was naturally confronted with the various notions connecting Weil's characters of type ($A_0$),  the {\it Gr\"o\ss encharaktere} of Hecke, and cohomological automorphic representations of ${\rm GL}(1)$, all under the general title of algebraic Hecke characters. 
The idea of writing these notes is so that anyone with a similar interest in the arithmetic theory of $L$-functions has a ready reference at hand with all the necessary details. There is nothing new in the first five sections other than a reorganization of one's thoughts to wade through the various dictionaries; this part of the article is based upon the references: Deligne \cite{deligne-sga}, \cite{deligne-corvallis}, 
Harder \cite{harder}, Hida \cite{hida-duke}, 
Neukirch \cite{neukirch}, Serre \cite{serre}, Schappacher \cite{schappacher}, 
Waldspurger \cite{waldspurger}, and Weil \cite{weil}. Also included are observations that play a crucial role in the study of 
special values of automorphic $L$-functions. 
The only pre-requisites to read the first part are basic algebraic number theory, the 
language of ad\`eles and id\`eles (for example, the first seven chapters of Weil's book \cite{weil-book} are sufficient), and some basic theory of sheaf cohomology.

\medskip

The last section (Sect.\,\ref{sec:harder-hecke}), which concerns the special values of Hecke $L$-functions, is conceptually and technically much deeper.  
In a pioneering paper, G\"unter Harder \cite{harder} laid the foundations of 
Eisenstein cohomology for $\GL(2)$ and as a consequence proved a rationality result on the ratios of critical values of Hecke $L$-functions. 
To give an idea of the shape of this result 
let $E$ and $F$ be number fields and $\chi$ an algebraic Hecke character of $F$ with coefficients in $E$. For an embedding of fields 
$\iota : E \to \C$ consider the $\C$-valued Hecke $L$-function $L(s, {}^\iota\chi)$. Suppose $m \in\Z$ is such that both $m$ and $m+1$ are {\it critical} for 
$L(s, {}^\iota\chi)$. This condition restricts $F$ to be a totally imaginary field that contains a maximal CM subfield. 
Let $\delta_{F/\Q}$ denote the absolute discriminant of $F$. 
Harder proved (\cite{harder}) that the complex number 
$|\delta_{F/\Q}|^{1/2}  \frac{L(m, {}^\iota\chi)}{L(m+1, {}^\iota\chi)}$
is algebraic (see \eqref{eqn:harder-1}) 
and satisfies a reciprocity law that it is equivariant under the Galois group $\Gal(\bar\Q/\Q)$ of $\Q$ (see \eqref{eqn:harder-2}); it follows then 
that the algebraicity statement can be strengthened to assert that the above quantity is in $\iota(E)$ (see \eqref{eqn:harder-3}).  
The {\it raison d'\^etre} for 
writing this article is that whereas the preliminary algebraicity result \eqref{eqn:harder-1} is correct, the reciprocity law \eqref{eqn:harder-2} is not correct as it stands by producing an example (see Sect.\,\ref{sec:counterexample}) when $F$ is a totally imaginary field that is not of CM type that gives a counterexample to 
\eqref{eqn:harder-3}. The real aim of this article is to state the correct reciprocity law which is stated in Thm.\,\ref{thm:main}. 
The missing signature, which is the term $\varepsilon_{{\bf n}, \iota}(\varsigma) \cdot 
\varepsilon_{\tilde{\bf n}, \iota}(\varsigma)$ in Thm.\,\ref{thm:main}, needed to rectify the reciprocity law is rather complicated, and turns out to be trivial when
$F$ is a CM field, and can be nontrivial for a general totally imaginary field. The presence of this subtle signature can be finessed using, interestingly enough, some other passages in \cite{harder}.  
Sect.\,\ref{sec:harder-hecke} is an exposition of some results in my paper \cite{raghuram-tot-imag} for Rankin--Selberg $L$-functions for $\GL_n \times \GL_{n'}$ over a totally imaginary base field 
in the simple situation of $n = n' = 1.$ 

\medskip

That there is a signature missing was suggested by calculations due to Pierre Deligne (see Prop.\,\ref{prop:deligne}),  
in the context of motivic $L$-functions, carried out from the point of view afforded by his celebrated conjecture--which is recalled in Conj.\,\ref{con:deligne} in the context at hand. 
After relating the missing signature  to the signature predicted by Deligne, the main theorem is restated as a much cleaner looking 
Thm.\,\ref{thm:main-restated}.  

\medskip
As a final comment to end the introduction, the results of this article show that whereas the analytic theory of $L$-functions of automorphic forms on $\GL(1)$ is not sensitive to the inner structure of the base field, the arithmetic theory of $L$-functions concerning the special values of $L$-functions crucially depends on the arithmetic structure of the base field. 

\medskip

{\small
{\it Acknowledgements:} 
I thank Pierre Deligne, G\"unter Harder, and Haruzo Hida for conversations and email correspondence on arithmetic aspects of Hecke characters and special values of $L$-functions; some of these discussions have found their way into this article. I acknowledge support from a MATRICS grant MTR/2018/000918 of the Science and Engineering Research Board, Department of Science and Technology, Government of India.}

\bigskip
\section{\bf Hecke characters}

We set up some general notation for this article: 

\begin{enumerate}
\item[] $\Q$ is the rational number field, with closure $\bar{\Q}$ inside complex numbers $\C$;
\item[] $F$ is a number field, i.e., a finite extension of $\Q$; \\
we do not assume that $F$ is contained in $\bar{\Q} \subset \C$; 
\item[] $d_F = d = [F:\Q]$;
\item[] $\O_F$ is the ring of integers of $F$; 
\item[] $U_F = \O_F^\times$ is the group of units of $\O_F$;
\item[] $\Sigma_F = {\rm Hom}(F,\C)$, all embeddings of $F$ into $\C$; note: ${\rm Hom}(F,\bar\Q) = {\rm Hom}(F,\C);$
\item[] $S_\infty$ = set of all archimedean places of $F$;
\item[] $S_\infty = S_r \cup S_c$ the real and complex places, respectively; 
\item[] $r_1 = \#S_r,$ $r_2 = \#S_c$; $d_F = r_1 + 2r_2;$
\item[] $\p$ a finite prime ideal of $\O_F$ or an infinite place;
\item[] $F_\p$ the completion of $F$ at $\p;$
\item[] $\O_\p$ the ring of integers of $F_\p$ at a finite prime $\p;$ 
\item[] $\varpi_\p$ is a uniformizer at a finite prime $\p;$ hence $\p \O_\p = \varpi_\p \O_\p;$ 
\item[] $U_\p = \O_\p^\times = \O_\p - \varpi_\p\O_\p$ is the group of units of $\O_\p$ for a finite prime $\p.$
\end{enumerate}

\medskip
\subsection{Class groups}
We start with some preliminaries on class groups, towards which we need some more notation: 
\begin{enumerate}
\item[] $\J_F$ is the group of all fractional ideals of $F;$
\item[] $\P_F$ is the group of all principal fractional ideals $(x) = x\O_F$ for $x \in F^\times;$  
\item[] $x \gg 0$ means that $x \in F$ is totally positive, i.e., $\rho(x) > 0$ for all $\rho \in \Hom(F,\R)$;
\item[] $\P_F^+ = \{ (x) \in \P_F \ : \ x \gg 0\};$
\item[] $\Cl_F = \J_F/\P_F$ is the class group of $F$;
\item[] $\Cl_F^+= \J_F/\P_F^+$ is the narrow class group of $F$.
\end{enumerate}
The adjective `narrow' is to suggest narrow conditions in $\P_F^+$ that we divide by to get $\Cl_F^+$; the narrow class group surjects onto the class group; we have the exact sequence: 
\begin{equation}
0 \longrightarrow \frac{\P_F}{\P_F^+} 
\longrightarrow \frac{\J_F}{\P_F^+} 
\longrightarrow \frac{\J_F}{\P_{F}} \longrightarrow 0. 
\end{equation}
Observe that 
$$
\frac{\P_F}{\P_F^+}  \ \simeq \ \frac{F^\times/U_F}{F^\times_+/U_F^+} \ \simeq \ 
\frac{F^\times}{F^\times_+ U_F}, 
$$
where $F^\times_+ = \{x \in F^\times  : x \gg 0\}$ and $U_F^+ = U_F \cap F^\times_+$ is the group of totally positive units. 
We have an exact sequence: 
\begin{equation}
0 \longrightarrow \frac{U_F}{U_F^+} 
\longrightarrow \frac{F^\times}{F^\times_+} \longrightarrow
\frac{F^\times}{F^\times_+ U_F} \longrightarrow 0. 
\end{equation}
Splicing the two exact sequences, we get a four-term exact sequence relating the narrow class group to the class group: 
\begin{equation}\label{eqn:4-term-ideal}
0 \longrightarrow \frac{U_F}{U_F^+} 
\longrightarrow \frac{F^\times}{F^\times_+} \longrightarrow
\Cl_F^+\longrightarrow \Cl_F \longrightarrow 0. 
\end{equation}
The narrow class group surjects onto the class group and the ratio of their orders is a power of $2$ since 
$[F^\times:F^\times_+] = 2^{r_1}.$ (This suggests an interesting exercise in basic number theory: 
given $r_1, t_1 \in \Z,$ with $r_1 \geq 1$ and $0 \leq t_1 \leq r_1$, construct a number field $F$ with exactly $r_1$ real embeddings and such that the order of the group $U_F/U_F^+$ is $2^{t_1},$ equivalently, the kernel of $\Cl_F^+\rightarrow \Cl_F$ is $2^{r_1-t_1}.$)

\medskip

To discuss the class groups in terms of id\`eles, here are some more notations: 
\begin{enumerate}
\item[] $\A_F$ is the ad\`ele ring of $F;$
\item[] $\I_F$ is the group of id\`eles of $F$;
\item[] $C_F = \I_F/F^\times$ is the group of id\`ele-classes of $F$; 
\item[] $U_\p := \O_\p^\times$ for finite $\p$ as before, but moreover \\ 
$U_\p := \R^\times_+$ for $\p \in S_r,$ and 
$U_\p := \C^\times$ for $\p \in S_c;$ 
\item[] $\I_F^S$ = group of $S$-id\`eles for any finite set $S$ of places; 
\item[] $F_\infty \ := \ F \otimes \R \ \simeq \ 
\prod_{v \in S_r} F_v  \times \prod_{w \in S_c} F_w \ \simeq \ 
\prod_{v \in S_r} \R \times \prod_{w \in S_c} \C;$
\item[] $F_{\infty +} = \{ x = (x_\lambda) \in F_\infty \ : \ x_v > 0, \ \forall v \in S_r \}.$
\end{enumerate}
There is a canonical map 
$\bfgreek{iota} : \I_F \to \J_F$
which sends an id\`ele $a = (a_\p)$ to the ideal $\bfgreek{iota}(a) :=  \prod_{\p \notin S_\infty} \p^{{\rm ord}_\p(a_\p)}.$ It is clear that
$
{\rm Kernel}(\bfgreek{iota}) = F_\infty^\times \prod_{\p \notin S_\infty} U_\p.
$
The map $\bfgreek{iota}$ induces an isomorphism which gives an id\`ele-theoretic description of the class group of $F$:
\begin{equation}
\frac{\I_F}{F^\times(F_\infty^\times \prod_{\p \notin S_\infty} U_\p)} \  \simeq \ \frac{\J_F}{\P_F} \ = \ \Cl_F. 
\end{equation}
Similarly, one has for the narrow class group: 
$$
\frac{\I_F}{F_+^\times(F_\infty^\times \prod_{\p \notin S_\infty} U_\p)} \ \simeq \ \frac{\J_F}{\P_F^+} \ = \ \Cl_F^+. 
$$
Weak-approximation gives that $F^\times$ is dense in $F_{\infty}^\times$; hence 
$F_+^\times F_\infty^\times = F^\times F_{\infty +}^\times;$ whence: 
\begin{equation}
\label{eqn:ray-class-group}
\frac{\I_F}{F^\times(F_{\infty +}^\times \prod_{\p \notin S_\infty} U_\p)} \ \simeq \ \frac{\J_F}{\P_F^+} \ = \ \Cl_F^+. 
\end{equation}

The four-term long exact sequence relating the class group and the narrow class group described using id\`eles takes the form: 
\begin{equation}\label{eqn:4-term-idele}
0 \longrightarrow \frac{U_F}{U_F^+} 
\longrightarrow \frac{F_\infty^\times}{F_{\infty +}^\times} \longrightarrow
\frac{\I_F}{F^\times(F_{\infty +}^\times \prod_{\p \notin S_\infty} U_\p)} 
\longrightarrow 
\frac{\I_F}{F^\times(F_\infty^\times \prod_{\p \notin S_\infty} U_\p)}  \longrightarrow 0. 
\end{equation}
To compare the second term from the left in (\ref{eqn:4-term-ideal}) with that in (\ref{eqn:4-term-idele}) we note that weak-approximation says that the canonical map 
$F^\times/F_+^\times \to F_\infty^\times/F^\times_{\infty +}$ 
induced by the diagonal inclusion $F^\times \hookrightarrow F_\infty^\times$ 
is an isomorphism.

\medskip
\subsection{Basics of Hecke characters} 
A Hecke character is a continuous homomorphism
$$
\chi \ : \ \I_F/F^\times \ \to \ \C^\times.
$$
We do not ask $\chi$ to be unitary. Some people prefer calling a unitary Hecke character as a Hecke character, and what we call a Hecke character, as a Hecke quasi-character. It is painful to keep writing `quasi-character' all the time, and so, with a possible abuse of terminology, we will work with the above definition. This is not so serious as we now explain: Consider the norm of an id\`ele $\alpha \in \I_F$ defined as $|\!| \alpha |\!| := 
\prod_\p |\alpha_\p|_\p$, where $\p$ runs over all valuations (finite and archimedean), and each valuation is normalized. 
Then $|\!| \ |\!|: \I_F \to \R^\times_+$ is a surjective homomorphism and let ${}^0\I_F$ denote its kernel. The product formula says that $|\!| a |\!| = 1$ for all $a \in F^\times$, i.e., 
$F^\times \subset {}^0\I_F.$ We have the following exact sequence: 
$$
0 \longrightarrow \frac{{}^0\I_F}{F^\times} \longrightarrow \frac{\I_F}{F^\times} \longrightarrow \R^\times_+ 
\longrightarrow 0.
$$
This sequence splits. There are several splittings; for example: map $t \in \R^\times_+$ into the id\`ele which is $t$ at a particular real infinite place and $1$ elsewhere, or map it to $t^{1/2}$ at a particular complex infinite place and $1$ elsewhere. But this depends on a choice of an infinite place. Instead, using $d = d_F = [F:\Q]$, 
we will always consider the following splitting: 
$$
t \mapsto (t^{1/d},\dots,t^{1/d}; 1,1,\dots), 
$$
where the id\`ele has the positive $d$-th root of $t$ at every infinite place and $1$ at every finite place. (Note that it is indeed a splitting, because for any complex place, the normalized valuation satisfies $|x|_\C = |x|_\R^2$ for $x \in \R.$) This splitting gives
$$
\I_F/F^\times \ \simeq \ {}^0\I_F/F^\times \times \R^\times_+.
$$
It is a fundamental fact that ${}^0\I_F/F^\times$ is compact (Neukirch~\cite[Theorem VI.1.6]{neukirch}). 
A continuous homomorphism of ${}^0\I_F/F^\times$ into $\C^\times$ has compact image and so lands in $S^1.$  
Further, any homomorphism $\R^\times_+ \to \C^\times$ is of the form $x \mapsto |x|^w$ for a complex number 
$w = \sigma + i \varphi$. (See \ref{sec:character-r} below.) Putting these remarks together, 
any Hecke character $\chi$ can be uniquely factored as  
\begin{equation}
\label{eqn:hecke-unitary}
\chi \ = \ {}^0\chi \otimes |\!| \ |\!|^{\sigma}
\end{equation}
for a unitary Hecke character ${}^0\chi : \I_F/F^\times \to S^1$ and $\sigma \in \R;$ sometimes ${}^0\chi$ is denoted also as $\chi^u.$

\medskip
\subsection{Dirichlet characters: Hecke characters of finite order}

A continuous homomorphism $\chi : \I_F/F^\times \rightarrow \C^\times$
with finite image and which is unramified everywhere (i.e., $\chi_\p := \chi|_{F_\p^\times}$ is trivial on the units $U_\p$ for all finite $\p$) gives a  
character of $\frac{\I_F}{F^\times(F_{\infty +}^\times \prod_{\p \notin S_\infty} U_\p)},$ and by \eqref{eqn:ray-class-group}, is a character of the 
narrow class group $\Cl_F^+.$ This generalizes by introducing some level structure giving us a description of Hecke characters of finite order in terms of 
characters of narrow ray class groups with level structure. Here is some more notation: 
\begin{enumerate}
\smallskip
\item[] $\m$ will be an integral ideal with prime-factorization $\m = \prod_{\p \notin S_\infty} \p^{m_\p};$

\smallskip
\item[] $U_\p(m_\p)$ will be defined as: 
$$
U_\p(m_\p) = \left\{\begin{array}{ll} 
1 + \p^{m_\p}, & {\rm if} \ \p \notin S_\infty, \  \p \, | \, \m, \\
U_\p, & {\rm if} \ \p \notin S_\infty, \  \p \not| \, \m, \\ 
\R^\times_+, & {\rm if} \  \p \in S_r, \\
\C^\times, & {\rm if} \  \p \in S_c.
\end{array}\right.
$$
\item[] $\U_F(\m) := \prod_\p U_\p(m_\p);$ 

\smallskip

\item[] $\U_{F,f}(\m) := \prod_{\p \notin S_\infty} U_\p(m_\p);$

\smallskip

\item[] $C_F(\m) := \U_F(\m)F^\times/F^\times$ is called the congruence subgroup mod $\m$ in $C_F;$ 

\smallskip

\item[] $C_F/C_F(\m) := \I_F/\U_F(\m)F^\times$ is the id\`ele-theoretic narrow ray class group mod $\m$;

\smallskip

\item[] $\J_F(\m)$ is the group of all fractional ideals relatively prime to $\m;$

\smallskip

\item[] $\P_F(\m)$ is the group of all principal fractional ideals $(x)$ with $x \equiv 1 \pmod{\m}.$ 

\smallskip

\item[] $\P_F^+(\m)$ is the subgroup of $\P_F(\m)$ consisting of those $(x)$ with $x \gg 0;$ 

\smallskip

\item[] $\Cl_F(\m) := \J_F(\m)/\P_F(\m)$ is the ray class group mod $\m$;

\smallskip

\item[] $\Cl_F^+(\m) := \J_F(\m)/\P_F^+(\m)$ is the narrow ray class group mod $\m$. 
\end{enumerate}

\medskip
\begin{prop}
\label{prop:narrow-ray-class-gp-mod-m}
The canonical homomorphism $w : \I_F \to \J_F$ induces an isomorphism:
$$
\frac{C_F}{C_F(\m)} \ \simeq \ \frac{\J_F(\m)}{\P_F^+(\m)}.
$$
\end{prop}

\begin{proof}
Define:
$$
\I_F(\m) \ := \ \{\alpha \in \I_F \ : \ \alpha_\p \in U_\p(m_\p),  \ \forall \p |\m\cdot \infty \}.
$$
Weak approximation implies: $\I_F = F^\times \cdot \I_F(\m).$ Hence 
$$
C_F = \I_F/F^\times = \I_F(\m)/(F^\times \cap \I_F(\m)). 
$$
Note that $\bfgreek{iota}$ maps $\I_F(\m)$ onto $\J_F(\m)$ and similarly, maps $F^\times \cap \I_F(\m)$ onto $\P_F^+(\m).$ Further, 
${\rm Ker}\left(\bfgreek{iota} : \I_F(\m) \to \J_F(\m)\right) \ = \ \U_F(\m)$. The rest is easy.
\end{proof}

A fundamental property of these congruence subgroups $C_F(\m)$ is captured by

\begin{prop}
\label{prop:finite-index-subgp}
For any integral ideal $\m$, the congruence subgroup $C_F(\m)$ is a subgroup of the id\`ele class group $C_F$ of finite index. Conversely, any subgroup of finite index of $C_F$ contains $C_F(\m)$ for some $\m.$ 
\end{prop}

\begin{proof}
See Neukirch~\cite[Proposition VI.1.8]{neukirch}. 
\end{proof}

The moral of the above discussion is that Hecke characters of finite order are exactly the characters of narrow ray class groups: 

\begin{cor}[Finite order Hecke characters]
\label{cor:fin-order-hecke}
Consider a continuous homomorphism 
$$
\chi \ : \ \I_F/F^\times \  \longrightarrow \  \C^\times
$$
whose image is of finite order. Then there is an integral ideal $\m$ such that $\chi$ factors as
$$
\xymatrix{
\I_F/F^\times = C_F \ar[rr]^{\chi} \ar[rd] & & \C^\times \\
 & C_F/C_F(\m) \ar[ru]
}$$
The smallest such $\m$ is called the conductor of $\chi$ and will be denoted $\f_\chi.$ 
\end{cor}

\medskip
A Dirichlet character of $F$ is a (necessarily unitary) character of a narrow ray class group of $F$ with level structure, 
i.e., it is a homomorphism 
$$
\chi : \J_F(\m)/\P_F(\m)^+ \ \to \ S^1,
$$
for some integral ideal $\m.$ By Prop.\,\ref{prop:narrow-ray-class-gp-mod-m}, Prop.\,\ref{prop:finite-index-subgp}, and Cor.\,\ref{cor:fin-order-hecke} it follows 
that Dirichlet characters are exactly the Hecke characters of finite order. The algebraic Hecke characters that we want to understand are in general not of finite order. 

\medskip
An easy exercise is to take $F = \Q$, $\m = N\Z$ to see that a Dirichlet character for $\Q$ modulo $N\Z$ is exactly a homomorphism 
$(\Z/N\Z)^\times \to \C^\times;$ such characters were considered by Dirichlet in his proof of infinitude of primes in arithmetic progressions; this exercise justifies the terminology.

\medskip
\subsection{Local components at finite places and conductor}

Fix any place $\p$ of $F$. Then $F_\p^\times$ embeds into $\I_F$ as $x_\p \mapsto (1,\dots, 1, x_\p,1\dots),$ 
the idele having $x_\p$ at the place corresponding to $\p$ and $1$ elsewhere. 
The restriction $\chi_\p$ of $\chi$ to this copy of $F_\p^\times$ is called the local character at $\p$. Sometimes we use $v$ for a place instead of $\p.$ 
For almost all finite $\p$, the character $\chi_\p$ is unramified, i.e., it is trivial on $U_\p$. This may be seen by considering the restriction of $\chi$ to $\prod_{\p \notin S_\infty} U_\p$, a compact and totally disconnected group the image of which under a continuous homomorphism into $\C^\times$ must be finite by the so-called `no small-subgroups argument'; see, for example, Bump~\cite[Exercise 3.1.1]{bump}. One can write $\chi = \otimes'_v \chi_v$ as a restricted tensor product. 
By the same token $\chi$ is trivial on $\U_{F,f}(\m)$ for some integral ideal $\m.$
The smallest ideal $\m$ such that $\chi$ is trivial on $\U_{F,f}(\m)$ is called conductor of $\chi$ and 
will be denoted $\f_\chi.$

\medskip
\subsection{The character at infinity of a Hecke character}

\subsubsection{Characters of $\R^\times$} \label{sec:character-r}

Any continuous homomorphism  $\chi : \R^\times \to \C^\times$ is of the form 
$$
\chi(x) \ = \ {\rm sgn}(x)^n \, |x|^{w} \ = \ \left(\frac{x}{|x|}\right)^n \, |x|^{w}, 
$$
with $n \in \{0,1\}$ and $w \in \C.$ Such a character $\chi$ is unitary if and only if $w = \i \varphi \in \i \R.$

\subsubsection{Characters of $\C^\times$}
For $z = x+\i y \in \C$, define $|z| := |z|_\R := \sqrt{x^2 + y^2},$ and the normalised valuation on 
$\C$ as $|z|_\C := |z|_\R^2 = x^2 + y^2.$ A continuous homomorphism  $\chi : \C^\times \to \C^\times$ is of the form 
$$
\chi(z)  \ = \ \left(\frac{z}{|z|}\right)^n \, |z|_\C^{w}, 
$$
with $n \in \Z$ and $w \in \C.$ If $z = r e^{\i \theta}$ with $r \in \R^\times_+$ and $\theta \in \R$, then $\chi(z) = r^{2w}e^{\i n \theta}. $
Such a character $\chi$ is unitary if and only if $w = \i \varphi \in \i \R.$

\subsubsection{Description of $\chi_\infty$}
For a Hecke character $\chi$, let $\chi_\infty = \chi\, |_{F_\infty^\times}$ where $F_\infty^\times \hookrightarrow \I_F.$ To describe this character at infinity explicitly, we introduce some more notation:
\begin{enumerate}
\item[] $\lambda$ is any infinite place; $\lambda \in S_\infty;$
\item[] $v$ is any real place; $v \in S_r$ and $F_v \simeq \R$ canonically;
\item[] $w$ is any complex place; $w \in S_c$ and $F_w \simeq \C$ non-canonically;
\item[] $|x_\infty|_\infty = \prod_\lambda |x_\lambda|_\lambda$ for $x_\infty \in F_\infty$ is the product of normalized valuations.
\end{enumerate}

Keeping in mind that a Hecke character factorizes as $\chi = {}^0\chi \otimes |\ |^\sigma$ (see (\ref{eqn:hecke-unitary})),  
one can write the character at infinity $\chi_\infty$ on $x_\infty \in F_\infty^\times$ as 
\begin{equation}\label{eqn:hecke-infinity}
\chi_\infty(x_\infty) = 
\left(\prod_{\lambda \in S_\infty}
 \left(\frac{x_\lambda}{|x_\lambda|}\right)^{n_\lambda}  |x_\lambda|_\lambda^{i \varphi_\lambda} \right) \, 
 |x_\infty|_\infty^\sigma.
\end{equation}
where $n_v \in \{0,1\}$, $n_w \in \Z$, $\varphi_\lambda \in \R,$ and $\sigma \in \R$.

\bigskip
\section{\bf \textit{Gr\"o\ss encharaktere} mod $\m$}

\medskip
\subsection{Definition} We follow the treatment in Neukirch~\cite[Section VII.6]{neukirch}, except that for us these characters can take values in $\C^\times$, i.e., they need not be unitary. Before getting started, some further notation that will be useful: 

\begin{enumerate}
\item[] $U_F^1(\m) := \{u \in U_F \ | \ u \equiv 1 \pmod{\m} \}$ are the units congruent to $1$ mod $\m$;

\smallskip

\item[] $U_F^1(\m)^+ := \{u \in U_F^1(\m) \ | \ u \gg 0\};$ 

\smallskip

\item[] $\O_F(\m) := \{ a \in \O_F\ | \ (a,\m) = 1\};$ nonzero integers in $F$ relatively prime to $\m;$

\smallskip

\item[] $F(\m) := \{ x \in F\ | \ (x,\m) = 1\};$ nonzero elements of $F$ relatively prime to $\m;$

\smallskip

\item[] $F^1(\m) := \{x \in F \ | \ x \equiv 1 \pmod{\m} \};$ elements of $F$ congruent to $1$ mod $\m$. 
\end{enumerate}

By a {\it Gr\"o\ss encharakter} mod $\m$  
we mean a homomorphism $\psi : \J_F(\m) \to \C^\times$ for which there exists a pair of characters $(\psi_f, \psi_\infty)$ with
$$
\psi_f : (\O_F/\m)^\times \to \C^\times, \ \ {\rm and} \ \ \psi_\infty : F_\infty^\times \to \C^\times, 
$$
such that for all $a \in \O_F(\m)$, we have
\begin{equation}
\label{eqn:grossen}
\psi((a)) \ = \ \psi_f(a \ {\rm mod}\  \m) \, \psi_\infty(a), 
\end{equation}
where in the left hand side $(a) := a\O_F$ is the principal ideal generated by $a$.
We call $\psi_f$ the finite part  and $\psi_\infty$ the infinite part of $\psi$; they satisfy the compatibility condition: 
if $u \in U_F = \O_F^\times$ is a unit then $\psi((u)) = 1$ and so we require: 
$$
\psi_f(u \ {\rm mod}\  \m) \, \psi_\infty(u) = 1.
$$
Further, if $\epsilon \in U_F^1(\m) := \{u \in U_F \ | \ u \equiv 1 \ {\rm mod}\ \m \}$ then the above requirement implies that 
$\psi_\infty(\epsilon) = 1.$  
The following proposition explains the relation between $\psi$ and $(\psi_f,\psi_\infty)$.

\begin{prop}
\label{prop:basic_gross}
A {\it Gr\"o\ss encharakter} $\psi$ mod $\m$ uniquely determines $\psi_f$ and $\psi_\infty$; indeed, the restriction of $\psi$ to  the group $\P_F(\m)$ uniquely determines $\psi_f$ and $\psi_\infty,$ and these satisfy the above compatibility conditions.  

Conversely, given homomorphisms $\psi_f' : (\O_F/\m)^\times \to \C^\times$ and $\psi_\infty' : F_\infty^\times \to \C^\times$ 
satisfying the compatibility condition $\psi_f'(u \ {\rm mod}\  \m) \, \psi_\infty'(u) = 1$ for all $u \in U_F$ 
(and so necessarily $\psi_\infty'(\epsilon) = 1$ for all $\epsilon \in U_F^1(\m)$), 
there exists a  {\it Gr\"o\ss encharakter} $\psi$ mod $\m$ such that $\psi_f = \psi_f'$ and $\psi_\infty = \psi_\infty'.$ 

Finally, if 
$\xi$ is another {\it Gr\"o\ss encharakter} mod $\m$ such that $\xi_f = \psi_f$ and $\xi_\infty = \psi_\infty$, then 
$\psi \xi^{-1}$ is a character of $\Cl_F(\m)$--the ray class group mod $\m.$
\end{prop}

\begin{proof}
Using $\xymatrix{F^1(\m) \ar@{->>}[r] & \P_F(\m)} \subset \J_F(\m)$ we see that a {\it Gr\"o\ss encharakter} $\psi$ mod $\m$ 
determines via restriction a character of $F^1(\m).$ By weak approximation, $F^1(\m)$ is dense in $F_\infty^\times$; thus $\psi$ restricted to 
$\P_F(\m)$ uniquely determines $\psi_\infty.$ 
Clearly, $\psi$ and $\psi_\infty$ uniquely determine $\psi_f$ by 
(\ref{eqn:grossen}). For the converse, given $\psi_f'$ and $\psi_\infty'$, define $\tilde\psi'$ on  $\O_F(\m)$ by 
$$
\tilde\psi'((a)) \ := \ \psi_f'(a \ {\rm mod}\ \m)\, \psi_\infty'(a).  
$$
The compatibility conditions say that $\tilde\psi'(a)$ is well-defined. This $\tilde\psi'$ extends uniquely to $F(\m)$, and being trivial on $U_F^1(\m),$ we get 
a character $\psi' : \P_F(\m) \to \C^\times$ as $\psi'((a)) = \tilde\psi'(a).$ Now induce $\psi'$ to $\J_F(\m)$, i.e., consider the induced representation
$$
{\rm Ind}_{\P_F(\m)}^{\J_F(\m)} (\psi').
$$
An easy exercise with Mackey theory says that this representation is a multiplicity free direct sum of characters $\psi$ of 
$\J_F(\m)$ such that $\psi|_{\P_F(\m)} = \psi'$. Moreover, fix such a $\psi$, then 
$$
{\rm Ind}_{\P_F(\m)}^{\J_F(\m)} (\psi') \ = \ 
\psi \otimes {\rm Ind}_{\P_F(\m)}^{\J_F(\m)} (1\!\! 1) \ = \ 
\bigoplus_{\chi} \psi \otimes \chi
$$
with $\chi$ running over the characters of $\J_F(\m)/\P_F(\m) = \Cl_F(\m)$ the ray class group modulo $\m$. Hence, if $\xi$ is a 
{\it Gr\"o\ss encharakter} with same finite and infinite parts as $\psi$ then $\xi = \psi \otimes \chi.$
\end{proof}

Following the German language, the plural of {\it Gr\"o\ss encharakter} will be {\it Gr\"o\ss encharaktere}.

\bigskip
\subsection{The correspondence between Hecke characters and {\it Gr\"o\ss encharaktere}}

The domain of definition of a {\it Gr\"o\ss encharakter} surjects onto the domain of definition of a Hecke character as in the 
following exact sequence: 

\begin{prop}
\label{prop:exact-seq-dictionary}
We have an exact sequence 
$$
1 \ \ 
\longrightarrow  \ \frac{F(\m)}{U_F^1(\m)} \ \ 
\stackrel{\varkappa}{\longrightarrow} \ \ 
\J_F(\m) \times (\O_F/\m)^\times \times \frac{F_\infty^\times}{U_F^1(\m)} \ \
\stackrel{\varrho }{\longrightarrow} \ \ 
\frac{\I_F}{F^\times \cdot \U_{F,f}(\m)} \ \ 
\longrightarrow \ \ 1, 
$$
where 
\begin{enumerate}
\smallskip
\item $\varkappa(a) = ((a)^{-1}, \ a \ {\rm mod}\ \m, \ a \ {\rm mod}\ U_F^1(\m))$ for all $a \in F(\m)$; 
\smallskip
\item $\varrho = \alpha \otimes \beta \otimes \gamma^{-1}$ where 
  \begin{enumerate}
  \smallskip
\item $\alpha: \J_F(\m) \to  \I_F/(F^\times \cdot \U_{F,f}(\m))$ is induced by the map $\J_F \to \I_F$ 
that sends a prime ideal $\q$ to 
the id\`ele that has $\varpi_\q$ at the place corresponding to $\q$ and $1$'s elsewhere; 
\smallskip
\item $\beta : (\O_F/\m)^\times \to \I_F/(F^\times \cdot \U_{F,f}(\m))$ is induced by the map that sends $a \in \O_F(\m)$ to the id\`ele 
that has $a$ at all infinite places and also at all finite places $\p$ not dividing $\m$, and is $1$ at all $\p | \m$; and 
\smallskip
\item $\gamma : F_\infty^\times/U_F^1(\m) \to \I_F/(F^\times \cdot \U_{F,f}(\m))$ is induced by the inclusion 
$F_\infty^\times \subset \I_F$; check that $U_F^1(\m)$ is mapped into $F^\times \cdot \U_{F,f}(\m).$
  \end{enumerate}
\end{enumerate}
\end{prop}

\medskip
See Neukirch~\cite[Prop.\,VII.6.13]{neukirch} for a proof but it is an instructive exercise to fix a proof for oneself. Let us denote an $a \in \I_F$ as $(a_\infty; \dots a_\p \dots; \dots a_\q \dots)$ for $\p|\m$ and $\q \! \not| \, \m$, and $[a]$ the class of the id\`ele $a$ in 
$\I_F/(F^\times \cdot \U_{F,f}(\m)).$ Note that: 
\begin{enumerate}
\item $\alpha(\q) = [(\dots 1 \dots; \ \dots 1 \dots; \ 1, \dots 1, \varpi_\q, 1, \dots)]$ for any prime ideal $\q \! \not |\, \m$. 
\item $\beta(a) = [((\dots 1 \dots; \ a^{-1},\dots, a^{-1}; \ \dots 1 \dots)].$
\item $\gamma(x_\infty) = [(x_\infty;  \ \dots 1 \dots;  \ \dots 1 \dots)].$
\end{enumerate}
The rest of the details are left to the reader. 

\medskip
A Hecke character $\chi$ trivial on $\U_{F,f}(\m)$ will be called a Hecke character mod $\m$; necessarily the conductor of $\chi$ divides 
$\m.$ From the exact sequence, the dictionary between {\it Gr\"o\ss encharaktere} mod $\m$ and Hecke characters mod $\m$  
is clear: Given a {\it Gr\"o\ss encharakter} $\psi$ mod $\m$, with $\psi_f$ and $\psi_\infty$ its finite and infinite parts, we get a character 
$\psi \otimes \psi_f \otimes \psi_\infty$ on $\J_F(\m) \times (\O_F/\m)^\times \times (F_\infty^\times/U_F^1(\m))$ and 
\eqref{eqn:grossen} says that this is trivial on the image of $\varkappa$ and hence we get a character $\chi = \chi_\psi$ 
on $\I_F/(F^\times \cdot \U_{F,f}(\m)),$ i.e., $\chi$ is a Hecke character mod $\m$. Conversely, given such a $\chi$, we inflate 
it via $\varrho$ to get $\psi \otimes \psi_f \otimes \psi_\infty = \chi \circ \varrho.$

\medskip
\subsection{The infinity type of a {\it Gr\"o\ss encharakter}}

Given a {\it Gr\"o\ss encharakter} $\psi = (\psi_f, \psi_\infty)$ mod $\m$, by using the above dictionary and \eqref{eqn:hecke-infinity}, one sees that 
the character at infinity, which is a continuous character $\psi_\infty : F_\infty^\times \to \C^\times,$ is of the form: 
$$
\psi_\infty(x_\infty) = 
\left(\prod_{\lambda \in S_\infty}
 \left(\frac{x_\lambda}{|x_\lambda|}\right)^{n_\lambda}  |x_\lambda|_\lambda^{i \varphi_\lambda} \right) \, 
 |x_\infty|_\infty^\sigma, 
$$
where $n_v \in \{0,1\}$, $n_w \in \Z$, $\varphi_\lambda \in \R$ and $\sigma \in \R$; furthermore, $\psi_\infty$ being trivial on $U_F^1(\m)$ imposes
certain restrictions on the infinity type, i.e., the collection $((n_\lambda, \varphi_\lambda)_\lambda, \sigma)$ of exponents, that can occur.

\begin{lemma}
\label{lem:inf-type-gros-char}
Let $\psi$ be a {\it Gr\"o\ss encharakter} mod $\m$ with the above notations for $\psi_\infty$. Then: 
\begin{enumerate}
\item[(i)] there are no restrictions on $n_\lambda$ for any $\lambda \in S_\infty$, and on $\sigma$; however,  
\item[(ii)] $2 \varphi_v = \varphi_w = \varphi$ (say), for all $v \in S_r$ and all $w \in S_c$. 
\end{enumerate}
Hence, the character at infinity takes the shape: 
$$
\psi_\infty(x_\infty) = 
\left(\prod_{\lambda \in S_\infty}
 \left(\frac{x_\lambda}{|x_\lambda|}\right)^{n_\lambda} \right) \, 
 |x_\infty|_\infty^{\sigma + i \varphi}.
$$
\end{lemma}

We adumbrate how one proves this lemma. Write the character at infinity as: 
$$
\psi_\infty(x_\infty) \ = \ 
\prod_{v \in S_r} {\rm sgn}(x_v)^{n_v} |x_v|^{\sigma + i \varphi_v} \cdot
\prod_{w \in S_c} x_w^{\sigma + n_w/2 + i \varphi_w/2} \, \bar{x}_w^{\, \sigma - n_w/2 + i \varphi_w/2} 
$$
For brevity, denote 
$
f_v := \sigma + i \varphi_v, \  f_w := \sigma + n_w/2 + i \varphi_w/2,  \ 
f_{\bar{w}} := \sigma - n_w/2 + i \varphi_w/2.
$ 
As $\psi_\infty$ is trivial on any $u \in U_F^1(\m)^+$ (totally positive global units that are $1$ mod $\m$) we get 
$$
\prod_{v \in S_r} v(u)^{f_{v}} \prod_{w \in S_c} w(u)^{f_w}  \bar{w}(u)^{f_{\bar{w}}} = 1.
$$
We will go through the rest of the proof, which uses the ingredients of the proof in Dirichlet's unit theorem, in detail in the proof of Lem.\,\ref{lem:purity} below 
(the main focus of this article is the notion of an algebraic Hecke character to which that lemma particularly applies) giving us the constraint: 
$2 f_v = f_w$, from which the above lemma follows.

\bigskip
\section{\bf Algebraic Hecke characters}

\bigskip
\subsection{Definition and purity}
There is a canonical map $\Sigma_F  \to S_\infty.$ Let $\lambda \in S_\infty$. 
If $\lambda = v \in S_r$, i.e., $\lambda$ is a real place $v$ then it corresponds to a unique 
real embedding $\tau_v: F \to \R$, and if $\lambda = w \in S_c$, i.e., $\lambda$ is a complex place $w$ then it corresponds to a conjugate pair of complex embeddings 
$\{\tau_w, \bar\tau_w\}$ with the understanding that the choice of $\tau_w : F \to \C$ is not canonical; we use $\tau_w$ to identify $F_w \simeq \C.$ 
Recall, as a reminder of our notational artifice, that 
$F_\infty = F \otimes \R \simeq \prod_{\lambda \in S_\infty} F_\lambda \simeq \prod_{v \in S_r} \R \times \prod_{w \in S_c} \C.$ One may write 
$x_\infty \in F_\infty$ simply as $x_\infty = (x_\lambda)_{\lambda \in S_\infty},$ or more elaborately as 
$x_\infty = ((x_v)_{v \in S_r}, (z_w)_{w \in S_c})$; both notations have their intrinsic virtue. 
Let $\chi$ be Hecke character of $F$, and $\chi_\infty$ its character at infinity; see \eqref{eqn:hecke-infinity}. We say that $\chi$ is an 
{\it algebraic Hecke character} if for every $\tau \in \Sigma_F$ there exists an integer $n_\tau$ such that for $x_\infty \in F_\infty^\times$ we have:
\begin{equation}
\label{eqn:algebraic-H-char}
\chi_\infty(x_\infty) \ = \ 
\prod_{v \in S_r} x_v^{n_{\tau_v}} \prod_{w \in S_c} z_w^{n_{\tau_w}} \bar{z}_w^{\, n_{\bar\tau_w}}.
\end{equation}
By separating, the right hand side into unitary and non-unitary parts, while noting that $z^p \bar{z}^q = \left(\tfrac{z}{|z|}\right)^{p-q} |z|_\C^{(p+q)/2}$,  
we have: 
$$
\chi_\infty(x_\infty) \ = \ 
\left(\prod_{v \in S_r} \left(\frac{x_v}{|x_v|}\right)^{n_{\tau_v}} \prod_{w \in S_c} \left(\frac{z_w}{|z_w|}\right)^{n_{\tau_w} - n_{\bar\tau_w}}\right) 
\left(\prod_{v \in S_r} |x_v|_v^{n_{\tau_v}} \prod_{w \in S_c} |z_w|_w^{(n_{\tau_w} + n_{\bar\tau_w})/2}\right).
$$
If we compare this with \eqref{eqn:hecke-infinity}, we see $v \in S_r,$ $w \in S_c,$ and $\lambda \in S_\infty$ that: 
$$
n_v \equiv n_{\tau_v} \ {\rm mod} \ 2, \ \ 
n_w = n_{\tau_w} - n_{\bar\tau_w}, \ \ 
\varphi_\lambda = 0, \ \ 
2 \sigma = 2 n_{\tau_v} =  n_{\tau_w} + n_{\bar\tau_w}
$$ 
Especially note the constraint: $2 n_{\tau_v} =  n_{\tau_w} + n_{\bar\tau_w}$ for all $v \in S_r,$ $w \in S_c.$ Now, let us suppose the 
algebraic Hecke character has modulus $\m.$ Then this constraint gets significantly tighter by the following:

\begin{lemma}[Purity]
\label{lem:purity}
For each $\tau \in \Sigma_F$, suppose we are given $n_\tau \in \Z$. Suppose for some integral ideal $\m$ of $F$ we have 
\begin{equation}
\label{eqn:trivial-uf1m}
\prod_{\tau \in \Sigma_F} \tau(u)^{n_\tau} \ = \ 1, \quad \forall u \in U_F^1(\m).
\end{equation}
Then, there exists ${\sf w} \in \Z$ such that 
\begin{enumerate}
\item[(i)] if $S_r \neq \emptyset$ then $n_\tau = {\sf w}$ for all $\tau \in \tau_F,$ and 
\item[(ii)] if $S_r = \emptyset$ then $n_{\gamma \circ \tau} + n_{\gamma \circ \bar\tau} = {\sf w}$ for all $\tau \in \Sigma_F$ and $\gamma \in \GalQ.$
\end{enumerate}
\end{lemma}

\begin{proof}
The proof uses the ingredients of the proof in Dirichlet's unit theorem. Note that $U_F^1(\m)$ has finite index in $U_F$. 
Enumerate the real embeddings of $F$ as $\{v_1,\dots,v_{r_1}\}$ and the complex embeddings as  
$\{w_1, \bar{w}_1, \dots, w_{r_2}, \bar{w}_{r_2}\}.$ 
With a minor abuse of notation, for brevity, write $\tau_{v_i} = v_i$, and $\tau_{w_j} = w_j.$ 
Let $\cH \subset \R^{r_1+r_2}$ be the hyperplane given by the sum of all coordinates being $0.$
The Minkowski map $\l : U_F \to \cH$ is given by:
$$
\l(u) \ = \ (\log|v_1(u)|_\R, \, \dots, \, \log|v_{r_1}(u)|_\R, \, \log|w_1(u)|_\C, \dots, \, \log|w_{r_2}(u)|_\C ).
$$
In \eqref{eqn:trivial-uf1m}, apply $\log|\ |_\C$  to get: 
\begin{multline}
\label{eqn:lin-dep-reln}
2n_{v_1} \log|v_1(u)|_\R + \cdots + 2n_{v_{r_1}} \log|v_{r_1}(u)|_\R + \\ 
(n_{w_1} + n_{\bar{w}_1}) \log|w_1(u)|_\C + \cdots +  (n_{w_{r_2}} + n_{\bar{w}_{r_2}} ) \log|w_{r_2}(u)|_\C \ = \ 0, 
\end{multline}
for all $u \in U_F^1(\m)^+.$
The proof of Dirichlet's unit theorem gives that $\Gamma = \l(U_F)$ is a lattice in $\cH$; hence $\Gamma_F^1(\m) := \l(U_F^1(\m))$ is also a lattice 
in $\cH$; whence, there exists $u_1,\dots,u_{t-1} \in U_F^1(\m)$ such that 
$\{v_1 := \l(u_1), \dots, v_{t-1} := \l(u_{t-1})\}$ is an $\R$-basis for $\cH,$ where, for brevity, we denote $t = r_1+r_2.$ 
Let us write $v_i = (a_{i1},\dots,a_{it})$ as a vector 
in $\R^t,$ and consider the $(t-1) \times t$ sized matrix $A = [a_{ij}].$ The rank of $A$ is $t-1$ and hence its nullity is $1.$ Since each $v_i \in \cH$ we know
$\sum_j a_{ij} = 0$, we deduce that for any $X \in \R^{t \times 1}$, if $AX = 0$ then all the coordinates of $X$ are equal. Now  
\eqref{eqn:lin-dep-reln} applied to $u_1, \dots, u_{t-1}$ gives a solution to the system of equations $AX = 0$ from which we get 
$
2n_{v_1} = \cdots = 2n_{v_1} = n_{w_1} + n_{\bar{w}_1} = \cdots = n_{w_{r_2}} + n_{\bar{w}_{r_2}}.$   
(The reader may pause here to note that the above argument proves Lem.\,\ref{lem:inf-type-gros-char}.) 
Let us denote the common value by ${\sf w}'.$ Now, take any $\gamma \in \GalQ$, apply $\gamma^{-1}$ to \eqref{eqn:trivial-uf1m} to get 
$$
1 \ = \ \prod_{\tau \in \tau_F} \gamma^{-1}(\tau(u))^{n_\tau} \ = \ 
\prod_{\tau \in \tau_F} \tau(u)^{n_{\gamma \circ \tau}}, 
$$
and we go through the name argument to get $n_{\gamma \circ \tau} + n_{\gamma \circ \overline{\tau}} = {\sf w}'$ for all $\tau \in \tau_F$ and $\gamma \in \GalQ.$ If $S_r \neq \emptyset$, then since $\GalQ$ acts transitively on $\Sigma_F$ we see that $n_\tau = n_{v_1}$ for all $\tau$; then 
take ${\sf w} = n_{v_1} = {\sf w}'/2$. If $S_r = \emptyset$ then take ${\sf w} = {\sf w}'.$  
\end{proof}

\medskip

Let $\chi$ be an algebraic Hecke character of $F$ mod $\m$ and with infinity type $(n_\tau)_{\tau \in \Sigma_F}.$ Then the exponents $n_\tau$ satisfy the 
conditions of Lem.\,\ref{lem:purity}. The integer ${\sf w}$ is called the {\it purity weight} of $\chi$. It is worth noting that when $F$ is totally real or totally imaginary then there is no constraint on the parity of the purity weight ${\sf w}$, but when $F$ has both real and complex embeddings then the purity weight ${\sf w}$ is necessarily an even integer. Note the following easy consequence of purity: 

\begin{cor}
\label{cor:alg_hecke_2_cases}
Let $\chi$ be an algebraic Hecke character of $F$ mod $\m$ and with infinity type $(n_\tau)_{\tau \in \Sigma_F}$ and purity weight ${\sf w}.$ 
\begin{enumerate}
\item[(i)] If $S_r \neq \emptyset$ then $\chi = \chi^\circ |\!| \ |\!|^{\sf w}$ for a Dirichlet character $\chi^\circ$ of $F$ mod $\m$.  
\item[(ii)] If $S_r = \emptyset$ then $\chi = \chi^u |\!| \ |\!|^{{\sf w}/2}$ for a unitary Hecke character $\chi^u$ of $F$ mod $\m$.
\end{enumerate} 
\end{cor}
In case (ii), let us observe that the character at infinity of $\chi^u$ has the following shape: 
$$
\chi^u_\infty(x_\infty) \ = \ \prod_{w \in S_c = S_\infty} \left(x_w/\bar{x}_w\right)^{n_w - {\sf w}/2}.
$$

\medskip
\subsection{The $L$-function of an algebraic Hecke character}

\subsubsection{Definition(s) of Hecke $L$-function}
\label{sec:defn-l-fn}
Hecke considered the class of {\it Gr\"o\ss encharaktere} mod integral ideals as the optimal class of characters for which one may study the analytic properties of $L$-functions, called Hecke $L$-functions, that we now define. 
Let $\chi$ be an algebraic Hecke character, and $\psi = (\psi_f, \psi_\infty)$ the associated {\it Gro\ss encharakter}. Suppose that $\m$ is the conductor of 
$\chi$, and so also $\m$ is the conductor of $\psi$. The (finite part of the) Hecke $L$-function of $\chi$ may be defined in two possible ways: 
\smallskip
\begin{enumerate}
\item As defined by Hecke in the form of a Dirichlet series: 
$$ 
L_f(s, \chi) \ = \ \sum_{\a} \frac{\psi(\a)}{N_{F/\Q}(\a)^s},
$$
the summation running over integral ideals $\a$ of $\O_F$ that are relatively prime to $\m$. 

\smallskip
\item As defined by Tate in the form of an Euler product: 
$$ 
L_f(s,\chi)\ = \ \prod_\p \, \left(1 - \chi_\p(\varpi_\p) N_{F/\Q}(\p)^{-s}\right)^{-1},
$$
the product is over prime ideals $\p$ of $\O_F$ that are relatively prime to $\m.$ 
\end{enumerate}

\medskip
Here, $N_{F/\Q}(\p)$ is defined as $|\O_F/\p|,$ and is multiplicatively extended to define $N_{F/\Q}(\a).$ 
One has the usual analytic properties: absolute convergence in a half-plane, analytic continuation, and functional equation. In the form of Hecke's definition, we refer the reader to Neukirch \cite{neukirch}; in the form of Tate's definition, we refer the reader to Tate's thesis \cite{tate} and Weil's book \cite{weil-book}.

\subsubsection{The coefficients of the $L$-function of an algebraic Hecke character}

The main reason to consider algebraic Hecke characters (called characters of type $A_0$ by Weil \cite{weil}) is the following proposition which 
puts us in the right context to study the arithmetic properties of Hecke $L$-functions. 

\begin{prop}
\label{prop:algebraic_values}
Let $\chi$ be an algebraic Hecke character, and $\psi = (\psi_f, \psi_\infty)$ the associated {\it Gro\ss encharakter} of modulus, say, $\m.$ 
Then the coefficients of the Dirichlet series giving the finite part of the Hecke 
$L$-function $L_f(s, \chi)$ are contained in a number field, i.e., there is a finite extension $E$ of $\Q$ such that $\psi(\a) \in E$ (resp., $\chi_\p(\varpi_\p)
\in E$) for all integral ideals $\a$ (resp., prime ideals $\p$) relatively prime to $\m.$  
\end{prop}

\begin{proof}
By \eqref{eqn:grossen}, it follows that for $a \in \O_F(\m)$, and hence also for $a \in F(\m)$, 
$\psi((a))$ takes values in a finite extension $E_1$ of $\Q$ inside $\bar{\Q}$, since $(\O_F/\m)^\times$ is a finite group, 
and since $\psi_\infty$ is algebraic \eqref{eqn:algebraic-H-char} we see that 
$\psi_\infty(a)$ takes values in the compositum of all the conjugates of $F$ inside $\bar{\Q}.$ Hence the values of $\psi$ on $\P_F(\m)$ lie in $E_1$. Since, 
$\P_F(\m)$ has finite index in $\J_F(\m)$ we see that the values of $\psi$ are in a finite extension $E$ of $E_1$ contained inside $\bar\Q.$ 
By Prop.\,\ref{prop:exact-seq-dictionary}, $\psi(\p) = \chi_\p(\varpi_\p)$ for all prime ideals $\p \! \not | \, \m$. 
\end{proof}

\medskip
\subsubsection{The rationality field of an algebraic Hecke character}

The smallest subfield of $\bar\Q$ that contains all the `values' of an algebraic Hecke character $\chi$, denoted  $\Q(\chi)$, 
is called its {\it rationality field}. More precisely, 
with notations as in the above proposition, suppose $\chi$ has conductor $\m$, we may define: 
\begin{multline}
\label{eqn:rationality-field}
\Q(\chi) \ = \ \Q(\{\chi(\p) : \mbox{for all prime ideals $\p$ relatively prime to $\m$}\}) \\ 
= \Q(\{\psi(\a) : \mbox{for all integral ideals $\a$ relatively prime to $\m$}\}).
\end{multline} 
As Weil comments in \cite{weil}, $\Q(\chi)$ need not contain $F$; consider, for example, a quadratic character of any number field $F$. 

\medskip
\subsubsection{The critical values of the $L$-function of an algebraic Hecke character}
Let us recall the shape of the functional equation satisfied by $L_f(s, \chi),$ in as much as we recall the precise recipe for the $\Gamma$-factors 
$L_\infty(s, \chi)$ necessary to define the completed $L$-function:
$$
L(s, \chi) \ := \ L_\infty(s, \chi) L_f(s, \chi), 
$$
where, if $\chi$ has infinity type $(n_\tau)_{\tau \in \Sigma_F}$, $n_\tau \in \Z$, satisfying the purity condition of Lem.\,\ref{lem:purity}, then we define:
$$
L_\infty(s, \chi)  \ = \ 
\prod_{v \in S_r} \Gamma_\R(s + n_v + \epsilon_v) 
\cdot \prod_{w \in S_c} \Gamma_\C\left(s + \frac{n_w + n_{\bar{w}}}{2} + \frac{|n_w - n_{\bar{w}}|}{2}\right), 
$$
where $\Gamma_\R(s) = \pi^{-s/2}\Gamma(s/2),$ $\Gamma_\C(s) = 2(2\pi)^{-s} \Gamma(s)$, and $\epsilon_v \in \{0,1\}$ with 
$\epsilon_v \equiv n_v$ mod $2$. When $\chi$ is nontrivial, the completed $L$-function admits an analytic continuation to an entire function
of $s \in \C$, and a beautiful aspect of $L$-functions is the functional equation:
$$
L(s, \chi) \ = \ \varepsilon(s, \chi) L(1-s, \chi^{-1}), 
$$
where $\varepsilon(s, \chi)$ is an exponential function. (See \cite{neukirch} or \cite{weil-book}.) 
An integer $m$ is said to be {\it critical} for the Hecke $L$-function $L(s, \chi)$ if 
both $L_\infty(s, \chi)$ and $L_\infty(1-s, \chi^{-1})$ are regular (i.e., finite) at $s = m.$
We may say that the $\Gamma$-factors on either side of the functional equation do not have poles at $s = m$. 
The set of all critical integers for $L(s,\chi)$, denoted $\Crit(L(s,\chi))$, is given by the following 

\begin{prop}
\label{prop:critical}
Let $\chi$ be an algebraic Hecke character of a number field $F$ with infinity type $(n_\tau)_{\tau \in \Sigma_F}$, $n_\tau \in \Z$, satisfying the purity condition of Lem.\,\ref{lem:purity} with purity weight ${\sf w}$. Then the critical set of integers for $L(s, \chi)$ is given by:
\begin{enumerate}
\item[(i)] $F$ is totally real ($S_c = \emptyset$): Recall that $n_\tau = {\sf w}$ for all $\tau$. 
\smallskip  
  \begin{enumerate}
  \item If there exists $v_1, v_2 \in S_r$, such that $\epsilon_{v_1} \neq \epsilon_{v_2}$ then $\Crit(L(s,\chi)) = \emptyset.$
  \smallskip  
  \item If $\epsilon_v = 0$ for all $v \in S_r$ then $\Crit(L(s,\chi))$ is given by: 
  $$ \{\dots, \, -{\sf w}+1-2k, \, \dots, \, -{\sf w}-3, \, -{\sf w}-1; \, -{\sf w}+2, \, -{\sf w}+4, \, \dots, \, -{\sf w}+2k, \, \dots \}_{k \in \Z_{\geq 0}}.$$ 
  The critical set is centered at $\tfrac12 - {\sf w}$.
 \smallskip  
  \item If $\epsilon_v = 1$ for all $v \in S_r$ then $\Crit(L(s,\chi))$ is given by:
  $$\{\dots,-{\sf w}-2k, \dots, -{\sf w}-2, -{\sf w}; \, -{\sf w}+1, -{\sf w}+3, \dots, -{\sf w}+1+2k,\dots \}_{k \in \Z_{\geq 1}}.$$ 
  The critical set is centered at $\tfrac12 - {\sf w}$.
  \end{enumerate}
\medskip  
\item[(ii)] $F$ is totally imaginary ($S_r = \emptyset$):  Define the width of $\chi$ as the non-negative integer: 
$\ell = \ell(\chi) = \min\{|n_w - n_{\bar{w}}| \, : \, w \in S_c\}.$ Then ${\sf w} \equiv \ell \ {\rm mod} \ 2.$ 
We have: 
$$\Crit(L(s,\chi)) \ = \ \left\{ m \in \Z \ : \ 1 - \frac{\sf w}{2} - \frac{\ell}{2} \ \leq \ m \ \leq \ - \frac{\sf w}{2} + \frac{\ell}{2} \right\}.$$
The critical set, centered at $\tfrac{1- {\sf w}}{2}$, is a finite set of cardinality $\ell$. 
\medskip  
\item[(iii)] If $F$ has real and complex places ($S_r \neq \emptyset \neq S_c$), then 
$\Crit(L(s,\chi)) = \emptyset.$
\end{enumerate}
\end{prop}

We omit the tedious proof of the above proposition and just add a few comments after which it is an extended exercise that is left to the reader. 
Recall that the poles of $\Gamma(s)$ are simple and located at non-positive integers, and furthermore that $\Gamma(s) \neq 0$ for all $s$. In case (i)(a), one sees the presence of a product 
$\Gamma(s/2)\Gamma((1+s)/2)$ because of which no integer can be critical. 
In case (i)(b), up to shifting by $-{\sf w},$  
it is the same as $\Crit(\zeta(s))$, the critical set of the Riemann zeta function. In case (i)(c), up to shifting by $-{\sf w},$ it is the same as the critical set of an odd classical Dirichlet character. For case (ii), the reader is referred to \cite[Sect.\,3.1]{raghuram-tot-imag} where the details are explicated. 
Case (iii), is like the intersection of cases (i) and (ii); by Lem.\,\ref{lem:purity}, $n_\tau = {\sf w}$ for all $\tau$; $\ell = 0$ where $\ell$ is defined using only $S_c$, forcing no critical points if there are real and complex places.  

\medskip
There has been a huge amount of work on the nature of the critical values of $L(m, \chi)$. Deligne's conjecture \cite{deligne-corvallis}
on the special values of motivic $L$-functions sets the stage for the sort of theorem one seeks. The interested reader is referred to 
Blasius \cite{blasius-annals}, Harder \cite{harder}, and Schappacher \cite{schappacher}, and all the references therein. My reason for 
writing this article is that while I was working on the critical values of automorphic $L$-functions (\cite{raghuram-CM} and \cite{raghuram-tot-imag}),  
I had to work through the results of this article frequently; and these results, although well-known to experts, did not seem to be readily available in the literature in the way I needed them.  There are some very deep sign problems, apparently hitherto not well-appreciated, that appear in these works which have their roots -- to paraphrase Weil~\cite{weil} -- in the remarkable properties of such algebraic Hecke characters.

\medskip
The above proposition begets the terminology of a {\it critical algebraic Hecke character} by which we mean an algebraic Hecke character $\chi$ such that 
$\Crit(L(s,\chi)) \neq \emptyset.$ Existence of a critical algebraic Hecke character $\chi$ of a number field $F$ implies first of all that $F$ is totally real or 
totally imaginary, and furthermore, when $F$ is totally real then the parities $\epsilon_v$ of all the local archimedean characters are equal. 
We will now discuss 
the existence of critical algebraic Hecke characters with prescribed infinity type subject to the purity condition.

\medskip
\subsection{Existence of (critical) algebraic Hecke characters}

For a given number field $F$ we discuss the problem of existence of algebraic Hecke characters that may or may not be critical. The purity lemma
naturally leads us to consider two cases (i) $S_r \neq \emptyset$ and (ii) $S_r = \emptyset$; we will see that case (ii) when $F$ is totally imaginary is 
in some sense far more interesting. 

\subsubsection{When $F$ has a real place ($S_r \neq \emptyset$)} 
From Cor.\,\ref{cor:alg_hecke_2_cases}, (i), we know that an algebraic Hecke character $\chi$ is necessarily of the form 
$$\chi \ = \ \chi^\circ |\!| \ |\!|^{\sf w},$$ for a character $\chi^\circ$ of finite order and for an integer ${\sf w}.$ We may fine-tune our question a little further to 
note that we can arrange for $\chi^\circ$ to have any prescribed signature at infinite places: for any (possibly empty) $S \subset S_r$, by weak approximation, we can find $x \in F^\times$ such that $v(x) > 0$ for $v \in S$, and $v(x) < 0$ for $v \in S_r \setminus S.$ Take $E = F(\sqrt{x})$, and 
$\chi = \omega_{E/F}$ the quadratic character of $F$ corresponding to $E$ by class field theory. Then $\chi_{v}(-1) = 1$ for $v \in S$ and 
$\chi_v(-1) = -1$ for $v \in S_r \setminus S.$ 
In particular, when $F$ is a totally real field, by taking $S = \emptyset$ or $S = S_r = S_\infty$, we
see that there exists critical algebraic Hecke characters of any prescribed parity (totally even or totally odd) 
and any prescribed purity weight ${\sf w}.$

\subsubsection{When $F$ is totally imaginary ($S_r = \emptyset$)} 
Consider an infinity type: ${\bf n} := (n_\tau)_{\tau \in \Sigma_F},$ $n_\tau \in \Z,$ satisfying the purity condition 
$n_{\gamma \circ \tau} + n_{\gamma \circ \bar\tau} = {\sf w}$ for all $\tau \in \Sigma_F$ and all $\gamma \in \GalQ$. 
The question we wish to address is whether there exists an algebraic Hecke character $\chi$ of $F$ with infinity type ${\bf n}.$ As 
Weil says in \cite{weil}, Artin pointed out to him that this is essentially an exercise in Galois theory. 

\medskip

First, consider the case when $F$ is a CM field, by which one means that $F$ is a totally imaginary quadratic extension of a totally real field $F^+.$ 
Let us note an easy 

\begin{lemma}
\label{lem:CM-field-embeddings}
Let $F$ be a totally imaginary quadratic extension of a totally real field $F^+$. 
Let $\tau : F \to \bar{\Q}$ and $\gamma \in \GalQ$ and let $\c$ stand for complex conjugation as an element of $\GalQ$.  
Then 
$
\gamma \circ \c \circ \tau \ = \  \c \circ \gamma \circ \tau,
$
i.e., complex conjugation and any automorphism of $\bar{\Q}$ commute on the image of a CM field. This also holds for a totally real field.  
\end{lemma}  

\begin{proof}
 Let $\tau_1 = \gamma \circ \c \circ \tau$ and $\tau_2 = \c \circ \gamma \circ \tau.$ Then it is easy to that $\tau_1|_{F^+} = \tau_2|_{F^+}$ (since $F^+$ is totally real). This means that $\tau_1 = \tau_2$ or $\tau_1 = \c \circ \tau_2$; if the latter, then $\gamma \circ \c \circ \tau =  \gamma \circ \tau$, whence 
 $\c \circ \tau =  \tau$ or that $\tau$ lands inside $\R$ which is not possible. Hence $\tau_1 = \tau_2.$ The last assertion for a totally real field is obvious. 
\end{proof}

For a CM field $F$, consider an infinity type ${\bf n} := (n_\tau)_{\tau \in \Sigma_F},$ $n_\tau \in \Z,$ satisfying the purity condition 
$n_{\gamma \circ \tau} + n_{\gamma \circ \bar\tau} = {\sf w}$. The above lemma gives that $n_{\gamma \circ \bar\tau} = n_{\overline{\gamma \circ \tau}}$. Hence, purity for a CM field is the same as asking $n_\tau + n_{\bar\tau} = {\sf w}$ for all $\tau \in \Sigma_F.$

\begin{prop}
\label{prop:CM-field}
Let $F$ be a totally imaginary quadratic extension of a totally real field $F^+$. Let ${\bf n} := (n_\tau)_{\tau \in \Sigma_F},$ $n_\tau \in \Z,$ satisfying the purity condition $n_\tau + n_{\bar\tau} = {\sf w}$ for all $\tau \in \Sigma_F.$ Then there exists an algebraic Hecke character (of some modulus) with that infinity type. 
\end{prop}

\begin{proof} 
If such a character exists, then the character at infinity would have the shape: 
$$
\chi_\infty(x_\infty) \ = \ \left(\prod_{w \in S_c} \left(\frac{x_w}{\bar{x}_w}\right)^{\tfrac{n_{\tau_w} - n_{\bar\tau_w}}{2}}\right) |x_\infty|_\infty^{{\sf w}/2}, 
$$
and also for some integral ideal $\m$, $\chi_\infty$ needs to be trivial on $U_F^1(\m).$ 
Put $p_w = n_{\tau_w} - n_{\bar\tau_w}$, and write $x_w$ in polar form as $r_w e^{i \theta_w}.$ Then we need to construct a unitary 
Hecke character $\chi^u : \I_F/F^\times \to S^1$, such that $\chi^u_\infty(x_\infty) = \prod_w e^{i p_w \theta_w},$ and $\chi^u_\infty$ 
needs to be trivial on $U_F^1(\m).$ We can then take $\chi = \chi_u |\!| \ |\!|^{{\sf w}/2}.$ Note that $F_\infty^\times = (S^1 \times \cdots \times S^1) \times (\R^\times_+ \times \cdots \times \R^\times_+).$ 
The units $U_F$ (via the proof of Dirichlet's unit theorem) intersects with the compact part in roots of unity $\mu_F$ in $F^\times$. Take an integral ideal 
$\m$ such that $U_F^1(\m) \cap \mu_F = \{1\}$; hence,  
$U_F^1(\m) \cap (S^1 \times \cdots \times S^1) = \{1\}$. Define a unitary character $\psi^u_\infty$ that maps  
$x_\infty$ to $\prod_w e^{i p_w \theta_w}.$ Then $\psi^u_\infty$ defines a character of $F_\infty^\times/U_F^1(\m).$ Take any character 
$\psi^u_f$ of $(\O_F/\m)^\times$, and by Prop.\,\ref{prop:basic_gross} we get a {\it Gr\"o\ss encharakter} $\psi^u$ mod $\m$ with the necessary infinity type. 
Take $\chi^u$ to be the corresponding Hecke character as in the paragraph after Prop.\,\ref{prop:exact-seq-dictionary}. 
\end{proof}

\medskip

Next, consider the case of a general totally imaginary field $F$. There is a maximal CM or totally real subfield $F_1$ of $F$ that plays an important role. 
Following Weil \cite{weil}, let $F_0$ be the largest totally real subfield of $F$. Then $F_0$ admits at most one totally imaginary quadratic extension contained inside $F_0$. If $F$ does admit such a CM subfield, denote it as $F_1$; and if there is no such extension inside $F$, then put $F_1 = F_0.$
A basic property of a pure infinity type of a totally imaginary field is that it is the base-change of a pure infinity type of $F_1$ as explained by the following 

\begin{prop}
\label{prop:descend-to-F_1}
Let $F$ be a totally imaginary field. Consider an infinity type ${\bf n} := (n_\tau)_{\tau \in \Sigma_F},$ $n_\tau \in \Z,$ satisfying the purity condition 
$n_{\gamma \circ \tau} + n_{\gamma \circ \bar\tau} = {\sf w}$ for all $\tau \in \Sigma_F$ and all $\gamma \in \GalQ.$ Then $n_\tau$, as a function of 
$\tau$ depends only on the restriction of $\tau$ to $F_1,$ i.e., there exists a necessarily pure infinity type ${\bf m} := (m_{\tau_1})_{\tau_1 \in \Sigma_{F_1}}$
for the field $F_1$, such that if $\tau \in \Sigma_F$ with $\tau_1 = \tau|_{F_1}$ then $n_{\tau}  = m_{\tau_1}.$ (In this situation, we may say 
that ${\bf n}$ is the base-change of ${\bf m}$.)   
\end{prop}

\begin{proof}
(
I learnt of the following proof from Deligne; it appears in a different context in \cite{deligne-raghuram}.) The purity condition $n_{\gamma \circ \tau} + n_{\gamma \circ \bar\tau} = {\sf w},$ for all 
$\tau \in \Sigma_F$ and all $\gamma \in \GalQ,$ also implies that 
$n_{\gamma \circ \tau} + n_{\overline{\gamma \circ \tau}} = {\sf w}$, from which deduce 
$$
n_{\gamma \circ \bar\tau} \ = \ n_{\overline{\gamma \circ \tau}}, \quad \forall \tau, \ \forall \gamma.
$$ 
Consider the function $n : \GalQ \times \Sigma_F \to \Z$, defined as $n(\gamma, \tau) = n_{\gamma \circ \tau}.$ For any $x \in \GalQ$ we have 
$n(\gamma x, \tau) = n(\gamma, x \tau).$ The above displayed equation reads
$$
n(\gamma, \c \tau) \ = \ n(\c , \gamma \tau). 
$$
Replacing $\tau$ by $\gamma^{-1}\tau$, we also have: 
$$
n(\c , \tau) \ = \ n(\gamma, \c \gamma^{-1} \tau) \ = \ n(\gamma \c \gamma^{-1}, \tau).
$$
Replacing $\tau$ by $\c \tau$ we get
$$
n(1, \tau) \ = \ n(\gamma \c \gamma^{-1}\c, \tau). 
$$
Now, take any $x \in \GalQ$, and we see using the above relations (for all $\gamma$ and $\tau$) that:
\begin{multline*}
n(x \gamma \c \gamma^{-1}\c x^{-1}, \tau) \ = \ 
n(x \gamma \c \gamma^{-1} x^{-1} x \c x^{-1}, \tau) \ = \ 
n((x \gamma) \c (x \gamma)^{-1}, x \c x^{-1} \tau) \ = \\ 
n(\c, x \c x^{-1} \tau) \ = \ n(x \c x^{-1}, \c \tau) \ = \ n(x \c x^{-1} \c, \tau) \ = \ n(1, \tau).
\end{multline*}
In other words, $n_{h \tau} = n_\tau$ for all $h$ in the normal subgroup $\cN$ of $\cG := \GalQ$ generated by $\{\gamma \c \gamma^{-1} \c : \gamma \in \GalQ\}$. 
Note that $\cG/\cN$ is the largest quotient in which $\c$ is central. 
We see that $\cG$ acts transitively on $\Sigma_F;$ the transitive action of $\cG$ on $\Sigma_{F_1}$ is via $\cG/\cN$ (Lem.\,\ref{lem:CM-field-embeddings});  
and $\cN$ acts transitively on the fiber over $\tau_1 \in \Sigma_{F_1}$ in the restriction map $\Sigma_F \to \Sigma_{F_1}.$ In particular, 
if $\tau, \tau' \in \Sigma_F$ with $\tau_1 = \tau|_{F_1} = \tau'|_{F_1}$, and suppose $h \in \cN$ is such that $h \tau = \tau'$, then 
$n_{\tau'} = n_{h \tau} = n_{\tau},$ and we call this common value as $m_{\tau_1}.$ 
\end{proof}

\medskip
We are now ready to state the basic fact about the existence of algebraic Hecke characters for a totally imaginary field in the following

\begin{prop}
\label{prop:exist_alg_hecke_tot_imag}
Let $F$ be a totally imaginary field. Consider an infinity type ${\bf n} := (n_\tau)_{\tau \in \Sigma_F},$ $n_\tau \in \Z,$ satisfying the purity condition 
$n_{\gamma \circ \tau} + n_{\gamma \circ \bar\tau} = {\sf w}$ for all $\tau \in \Sigma_F$ and all $\gamma \in \GalQ.$ 
Then, there exists an algebraic Hecke character of infinity type ${\bf n}.$ 

\smallskip
\begin{enumerate}
\item[(i)] If $F_0 = F_1$ (this is the case when $F$ has no CM subfield), then there exists $n \in \Z$ such that $n_\tau = n$ for all $\tau \in \Sigma_F$. Any 
algebraic Hecke character $\chi$ of infinity type ${\bf n}$ is of the form $\chi = \chi^\circ |\!|\ |\!|^n$ for a Dirichlet character $\chi^\circ$ of $F$. 

\smallskip
\item[(ii)] If $F_1$ is a CM field, then ${\bf n}$ is the base-change of an infinity type ${\bf m}$ for $F_1$ (Prop.\,\ref{prop:descend-to-F_1}); any 
algebraic Hecke character $\chi$ of $F$ of infinity type ${\bf n}$ is of the form
$$
\chi \ = \ \chi_1 \circ N_{F/F_1} \otimes \chi^\circ, 
$$
for some algebraic Hecke character $\chi_1$ of $F_1$ with infinity type ${\bf m}$ and some Dirichlet character $\chi^\circ$ of $F$. 
\end{enumerate}
\end{prop}

\begin{proof}
In both cases from Prop.\,\ref{prop:descend-to-F_1} there exists an infinity type ${\bf m}$ over $F_1$ whose base-change to $F$ is ${\bf n}$. 
In case (i), since 
$F_1 = F_0$ is totally real, all the $m_{\tau_1}$ are equal, and hence $n_\tau$ are all equal, to say, $n$. If there is an algebraic Hecke character $\chi$ of 
this infinity type then necessarily we have $\chi  |\!|\ |\!|^{-n}$ is a finite-order character, say, $\chi^\circ.$ To construct $\chi$, take any finite-order Hecke character 
$\chi^\circ$ of $F$, and put $\chi = \chi^\circ |\!|\ |\!|^n$. 
In case (ii), from Prop.\,\ref{prop:CM-field}, there exists an algebraic Hecke character $\chi_1$ of $F_1$ with infinity type ${\bf m}.$ Consider
its base-change $\chi_2 := \chi_1 \circ N_{F/F_1},$ to $F$; it is clear that this base-change has the same infinity type as $\chi$, if one knew the existence of $\chi$, and in which case $\chi \chi_2^{-1}$ would be a character of finite-order, say, $\chi^\circ.$ Going backwards, construct the required $\chi$ via 
$\chi_1 \circ N_{F/F_1} \otimes \chi^\circ.$ 
\end{proof}

\medskip

I have found it helpful to keep some examples in mind while thinking of a totally imaginary field $F$, with various possible scenarios for $F_0$ and $F_1$. Some of these examples were suggested to me by Haruzo Hida. 

\begin{exam}{\rm 
When $F$ is totally imaginary but not CM, then one has to be careful about what constitutes a pure infinity type. The following example is instructive: 
 take $F = \Q(2^{1/3},\omega),$ where $2^{1/3}$ is the real cube root of $2$ and $\omega = e^{2\pi i/3}.$ Then $\Sigma_F = \Gal(F/\Q) \simeq S_3$ the permutation group in 3 letters taken to be $\{2^{1/3}, \, 2^{1/3}\omega, \, 2^{1/3}\omega^2\}.$ Let $s \in S_3$ correspond to $\tau_s : F \to \C.$ Consider the 
 infinity types ${\bf n} = (n_{\tau_s})_{s \in S_3}$ and ${\bf n}' = (n'_{\tau_s})_{s \in S_3}$: 
$$
\begin{array}{|c||c|c|c|c|c|c||} 
\hline
s & e & (12) & (23) & (13) & (123) & (132) \\ \hline 
n'_{\tau_s} & a & b & {\sf w}-a & c & {\sf w}-c & {\sf w}-b \\ \hline 
n_{\tau_s} & a & {\sf w}-a & {\sf w}-a & {\sf w}-a & a & a \\ \hline
\end{array}
$$
where $a,b,c,{\sf w} \in \Z$. For the tautological embedding $F \subset \bar\Q,$ the set $\Sigma_F$ is paired into complex conjugates as 
$\{(\tau_{e}, \tau_{(23)}), (\tau_{(12)}, \tau_{(132)}), (\tau_{(13)}, \tau_{(123)})\}.$ Then, $n'_\tau + n'_{\bar\tau} = {\sf w}.$ However, 
${\bf n}'$ is not a pure infinity type in general. All other possible pairings of $\Sigma_F$ into conjugates via automorphisms of $\bar\Q$ 
($F$ being Galois this simply boils down to composing these embeddings $\tau_s$ by a fixed one $\tau_{s_0},$ and using 
$\tau_{s_0} \circ \tau_{s} = \tau_{s_0s}$) are given by: 
$\{(\tau_{e}, \tau_{(12)}), (\tau_{(23)}, \tau_{(123)}), (\tau_{(13)}, \tau_{(132)})\},$ and  
$\{(\tau_{e}, \tau_{(13)}), (\tau_{(23)}, \tau_{(132)}), (\tau_{(12)}, \tau_{(123)})\},$
from which we see that ${\bf n}'$ is not a pure infinity type unless ${\sf w}-a, b$ and $c$ are all equal; ${\bf n}$ is pure and has purity weight ${\sf w}.$ 
Also, ${\bf n}$ is the base-change of the infinity type ${\bf m}$ of $F_1 = \Q(\omega)$ where $m_{\tau_e} = a$ and $m_{\tau_{(12)}} = {\sf w}-a.$ 
}\end{exam}

\begin{exam}{\rm 
Take a CM field $F^\cm$, and a number field $L$ that is linearly disjoint from $F^\cm$, and put $F = L \cdot F^\cm.$ 
This is a particularly pleasant example of a totally imaginary field. If $L$ is totally real, then $F$ is also of CM type. 
For concreteness, take $n \geq 3$, then $F = \Q(2^{1/n}, \zeta_n)$, where $\zeta_n = e^{2\pi i/n},$ is a totally imaginary field of this form; $F^\cm = \Q(\zeta_n)$ and $L = \Q(2^{1/n}).$ This example is slightly more general than the field considered in the previous example. 
}\end{exam} 

\begin{exam}\label{exam:hida}
{\rm 
For a totally imaginary field not of the above form, take $F_0 = \Q$, and put $F_1 = \Q(\i)$ and $F = \Q(\sqrt{4+\i}, \i).$ (This example is used in 
Sect.\,\ref{sec:counterexample}.)
}\end{exam}

\begin{exam}{\rm 
For a totally imaginary field with no CM subfield, take $F_0$ to be a cubic totally real field; for example, $F_0 = \Q(\zeta_3 + \zeta_3^{-1}).$ 
Choose non-squares $a$  and  $b$ in $F_0$,  with conjugates  $a, a', a''$  and  $b, b', b''$, respectively, and chosen so that 
$ a>0, a'<0, a''<0$ and $b<0, b'<0, b''>0.$  Such a choice is possible by weak approximation. 
Then the field $F= F_0[\sqrt{a},\sqrt{b}]$  is totally imaginary with no CM subfield. Such fields do not support critical algebraic Hecke characters since 
the width of a pure infinity type is $0$ (see Prop.\,\ref{prop:critical}).
}\end{exam}

\begin{rem}{\rm 
There is a slightly different way to organise one's thoughts about the subfields $F_0$ and $F_1$ of a totally imaginary field $F$, and which 
at times is useful to keep in mind. 
Recall that $\c$ denote the complex conjugation on $\C$. 
Let $F_1$ be maximal among all subfields $K$ of $F$ that have an involution $\iota_K : K \to K$ such that $\c \circ \sigma = \sigma \circ \iota_K$ for 
all embeddings $\sigma : K \to \bar\Q.$ In particular, $F_1$ has its involution $\iota_{F_1},$ and let $F_0$ 
denote the subfield of $F_1$ that is fixed by $\iota_{F_1}.$ There are two cases: 
(i) $\iota_{F_1}$ is non-trivial; then $F_1$ is a totally imaginary quadratic extension of the totally real $F_0;$ 
(ii) $\iota_{F_1}$ is trivial; then $F_1 = F_0$ is a totally real field. 
}\end{rem}

\bigskip
\section{\bf Cohomological automorphic representations of $\GL_1$}

So far, we have treated Hecke characters as complex valued functions on ad\`elic spaces, and one might say that this is from the point of view of 
{\it automorphic forms} on $\GL(1).$ We may also treat them as {\it algebro-geometric} objects; see, for example, Deligne \cite{deligne-sga} or Serre \cite{serre}. Of course, these points of view are inter-related for which the reader is recommended to see Schappacher \cite[Chap.\,0]{schappacher} for a lucid summary. In this section we consider Hecke characters from the point of view often taken in the {\it cohomology of arithmetic groups}. 
We begin by setting up the context, adopting the treatment in \cite[Chap.\,2]{harder-raghuram-book}.

\bigskip
\subsection{Locally symmetric spaces for $\GL_1/F$}
Let $G = \Res_{F/\Q}(\GL_1/F)$. Then: 

\smallskip
$G(\Q) = \GL_1(F) = F^\times;$ 

$G(\A) = \GL_1(\A_F) = \I_F;$ 

$G(\A_f) = \A_{F,f}^\times$ the group of finite id\`eles;  

$G(\R) = \GL_1(F \otimes \R) = F_\infty^\times = \prod_{v \in S_r} \R^\times \times \prod_{w \in S_c} \C^\times;$ 

$G(\R)^\circ = F_{\infty +}^\times$, the connected component of the identity in $G(\R);$

$K_\infty = \prod_{v \in S_r} \{\pm 1\} \times \prod_{w \in S_c} S^1,$ the maximal compact subgroup of $G(\R);$ 

$K_\infty^\circ \simeq \prod_{w \in S_c} S^1$, the connected component of the identity in $K_\infty;$

$A = \GL_1/\Q \hookrightarrow G/\Q$ is the canonical inclusion in $\Hom(\GL_1/\Q, \Res_{F/\Q}(\GL_1/F));$

$A(\R)$ is the diagonal copy of $\R^\times$ inside $G(\R) = \prod_{v \in S_r} \R^\times \times \prod_{w \in S_c} \C^\times;$

$A(\R)^\circ = \R_+^\times$, the connected component of the identity in $A(\R);$

$\m = \prod_\p \p^{m_\p}$, an integral ideal of $F$, which gives us a level structure, by which we mean:

$K_\m \subset G(\A_f),$ an open compact subgroup: $K_\m := \U_{F,f}(\m) = \prod_{\p | \m} (1+\p^{m_\p}) \times \prod_{\p \not\,| \m} U_\p.$ 

\medskip
\noindent

There are various quotients of $\I_F$ that one may consider depending on how much to divide by inside $G(\R)$. Here are five possible situations: 

\begin{subequations}
\label{eqn:loc-symm-space}
\begin{align}
S^G_\m & \ := \ G(\Q) \backslash G(\A) / K_\m.  \\
X^G_\m & \ := \ G(\Q) \backslash G(\A) / A(\R)^\circ K_\m.  \\ 
Y^G_\m  & \ := \ G(\Q) \backslash G(\A) / K_\infty^\circ K_\m.  \\ 
T^G_\m & \ := \ G(\Q) \backslash G(\A) / A(\R)^\circ K_\infty^\circ K_\m.  \\
C^G_\m & \ := \ G(\Q) \backslash G(\A) / G(\R)^\circ K_\m. 
\end{align}
\end{subequations}

\smallskip
Of course, $G$ being abelian, the double-coset-spaces are all quotients of $\I_F$; for example
$$
C^G_\m \ = \ G(\A)/(G(\Q) G(\R)^\circ K_\m) \ = \ \I_F/(F^\times \U_F(\m)) \ \approx \ \Cl^+_F(\m),
$$   
the narrow ray class group modulo $\m.$ Writing as a double-coset space is simply to suggest how things would generalize if $G$ is taken to be a general reductive group over $\Q$.  There are the canonical projection maps:
\begin{equation}
\xymatrix{
 & S^G_\m \ar[rd] \ar[ld] & \\
X^G_\m \ar[rd]  & & Y^G_\m \ar[ld] \\
 &  T^G_\m \ar[d] & \\
 & C^G_\m \approx \Cl^+_F(\m) & 
}
\end{equation}
The bottom most is a finite-abelian group being a narrow ray class group. 
To get a feel for the space $S^G_\m$ at the top, think of 
$$
G(\A)/ K_\m \ = \ G(\R) \times (G(\A_f)/K_\m),
$$ 
as the product of $G(\R)$ with the totally disconnected space $G(\A_f)/K_\m$; 
on which the discrete subgroup $G(\Q)$--it is discrete in $G(\A)$--acts diagonally, and taking the quotient by that action we have 
$G(\Q) \backslash (G(\R) \times (G(\A_f)/K_\m)).$ Note
that $G(\Q)\backslash G(\A_f)/K_\m$ is a finite set: 
$$
G(\Q)\backslash G(\A_f)/K_\m \ = \ \A_{F,f}^\times/(F^\times \U_{F,f}(\m)) \ \hookrightarrow \ \I_F/(F^\times (F_{\infty +}^\times \U_{F,f}(\m))) \ = \ \Cl_F^+(\m).
$$
There exists a finite set $\{x_{1,f}, \dots, x_{k,f}\}$ of finite id\`eles such that $\A_{F,f}^\times = \cup_{i=1}^k x_{i,f} F^\times \U_{F,f}(\m).$
Using these representatives a bijection 
$$
S^G_\m \ \simeq \ \coprod_{i=1}^k (F_\infty^\times/U_F^1(\m))x_{i,f}   
$$
can described as follows: for $x = (x_\infty, x_f) \in \I_F = F_\infty^\times \times \A_{F,f}^\times$, there is a unique $i$ such that $x_f = x_{i,f} \, \gamma \, u_f$ for 
some $\gamma \in F^\times$ and $u_f \in \U_{F,f}(\m)$; the bijection is given by: 
$$
S^G_\m \ni x(F^\times \U_{F,f}(\m)) \ \mapsto \ (\gamma^{-1}x_\infty U_F^1(\m), x_{i,f}) \in (F_\infty^\times/U_F^1(\m))x_{i,f}. 
$$
Let us describe the connected components of $S^G_\m$, which are all copies of  
$$
\frac{F^\times_{\infty +}}{U_F^1(\m)} \ \simeq \ \frac{(\R^\times_+)^{r_1} \times (\C^\times)^{r_2}}{U_F^1(\m)}.
$$ 
Let ${}^0 F^\times_{\infty +}$ denote the kernel of norm at infinity $|\!| \ |\!|_\infty : F^\times_{\infty +} \to \R^\times_+.$ Then
$$
F^\times_{\infty +} \ = \ {}^0 F^\times_{\infty +} \times A(\R)^\circ.
$$ 
By the product formula, $F^\times_{\infty +}$ contains $U_F^1(\m)$. Note that if 
$u \in U_F^1(\m) \cap A(\R)^\circ$ then $\sigma(u) = \sigma'(u)$ for all $\sigma, \sigma' \in \Sigma_F$, hence $u \in \Q^\times$; whence $u \in \{\pm 1\}.$ 
Assume the mild condition that $U_F^1(\m) \cap \{\pm 1\} = \{1\}.$ Hence we get
$$
\frac{F^\times_{\infty +}}{U_F^1(\m)} \ = \  A(\R)^\circ \times \frac{{}^0 F^\times_{\infty +}}{U_F^1(\m)} .
$$ 
Next, under the identification $F^\times_{\infty +} \simeq (\R^\times_+)^{r_1} \times (\C^\times)^{r_2}$, after writing $\C^\times = \R^\times_+ \times S^1$, and noting that $K(\R)^\circ = (S^1)^{r_2}$ which is the product of the $S^1$'s coming from the complex places, we have: 
$$
{}^0 F^\times_{\infty +} \ \simeq \ K(\R)^\circ \times {}^0(\R^\times_+)^{r_1+r_2},  
$$
where, ${}^0(\R^\times_+)^{r_1+r_2}$ consists of tuples $(t_v)_{v \in S_\infty}$ such that $\sum_v \log|t_v|_v = 0.$ 
Assume furthermore that $\m$ is deep enough so that $U_F^1(\m) \cap K(\R)^\circ$, which {\it a priori} is finite being discrete and compact in $F^\times_{\infty +}$, 
is in fact trivial, i.e., it amounts to assuming that $U_F^1(\m)$ has no roots of unity. Then, by the proof of Dirichlet's unit theorem using the Minkowski map 
(see the notations in the proof of Lem.\,\ref{lem:purity}) one has: 
$$
\frac{{}^0(\R^\times_+)^{r_1+r_2}}{U_F^1(\m)}
\approx
\frac{\cH \approx \R^{r_1+r_2-1}}{\Gamma_F^1(\m)} 
\approx 
(S^1)^{r_1+r_2-1}. 
$$
To summarise, any connected component of $S^G_\m$, for $\m$ deep enough, is homeomorphic to 
$$
A(\R)^\circ \times K(\R)^\circ \times (S^1)^{r_1+r_2-1},
$$
which in turn is isomorphic to $\R^\times_+ \times (S^1)^{r_1+2r_2-1}.$ 
Note that 
$$\dim(S^G_\m) = r_1+2r_2 = d_F = [F:\Q].$$
Similarly, the other spaces in \eqref{eqn:loc-symm-space} can be described. We can grant ourselves the license to call 
any of these as a `locally symmetric space' for $\GL_1$ over $F$ with level $\m.$

\bigskip
\subsection{Sheaves on locally symmetric spaces}

Take a field $E,$ which is assumed to be a finite Galois extension of $\Q$ that takes a copy of $F$; hence 
$\Hom(F, E) = \Hom(F, \bar{E}).$ 
This field $E$ may be called the field of coefficients. Consider an infinity type ${\bf n} = (n_\tau)_{\tau : F \to E}$ with $n_\tau \in \Z;$ one may also denote this as  
${\bf n} \in \Z[\Hom(F,E)]$.  
This is like an infinity type from before, except that it is now parametrised over $\Hom(F,E)$ instead of $\Hom(F, \bar{\Q})$ as before. 
An embedding $\iota : E \to \bar\Q$ gives an identification $\iota_* : \Hom(F,E) \to \Hom(F,\bar\Q)$ as 
$\tau \mapsto \iota_*\tau = \iota \circ \tau.$ Such an ${\bf n}$ gives us an algebraic representation $\vartheta_{\bf n}$ on a one-dimensional 
$E$-vector space $\cM_{{\bf n}, E}$ denoted: 
$$
\vartheta_{\bf n} : G \times E \ \longrightarrow \ \GL_1(\cM_{{\bf n},E}). 
$$
Since $G \times E = \Res_{F/\Q}(\GL_1/F)\times E = \prod_{\tau : F \to E} \GL_1 \times_{F, \tau} E,$ and on the component indexed by $\tau$, the 
representation is $x \mapsto \vartheta_{n^\tau}(x) := x^{n_\tau}$ for $x \in \GL_1(E).$ 
For $a \in G(\Q) = F^\times$, we have  
$$
\vartheta_{\bf n}(a) \ = \ \prod_{\tau : F \to E} \vartheta_{n_\tau}(\tau(a)) \ = \ \prod_{\tau : F \to E} \tau(a)^{n_\tau}. 
$$

\medskip

Given an infinity type ${\bf n} \in \Z[\Hom(F,E)]$, and the associated representation $(\vartheta_{\bf n}, \cM_{{\bf n},E})$, there is a sheaf $\tM_{{\bf n},E}$ of $E$-vector spaces on $S^G_\m$ constructed as follows. Consider the canonical projection map of dividing by $F^\times$: 
$$
\pi \ : \ \I_F/\U_{F,f}(\m) \ \longrightarrow \ F^\times\backslash \I_F/ \U_{F,f}(\m) = S^G_\m. 
$$
For any open subset $V \subset S^G_\m$, define the sections of $\tM_{{\bf n},E}$ over $V$ as the $E$-vector space:
\begin{equation}
\label{eqn:sheaf}
\tM_{{\bf n},E}(V) \ = \ \{s : \pi^{-1}(V)  \to \cM_{{\bf n},E} \ : \ s(a x) = \vartheta_{\bf n}(a) s(x), \ \ \forall a \in F^\times, \ \forall x \in \pi^{-1}(V)\}. 
\end{equation}
Checking the sheaf axioms is a routine exercise. The purity lemma (Lem.\,\ref{lem:purity}) can be recast as the following lemma: 

\begin{lemma}[Purity]
\label{lem:purity-sheaf}
Given ${\bf n} \in \Z[\Hom(F,E)]$, the sheaf $\tM_{{\bf n},E}$ of $E$-vector spaces on $S^G_\m$ is non-zero if and only if: 
there exists ${\sf w} \in \Z$ such that 
\begin{enumerate}
\item[(i)] if $S_r \neq \emptyset$ then $n_\tau = {\sf w}$ for all $\tau \in \Hom(F,E),$ and 
\item[(ii)] if $S_r = \emptyset$ then $n_{\iota \circ \tau} + n_{\overline{\iota \circ \tau}} = {\sf w}$ for all $\tau \in \Hom(F,E)$ and for all 
$\iota \in \Hom(E,\bar\Q).$
\end{enumerate}
\end{lemma}

\begin{proof}
A sheaf is nonzero if and only if it admits a nonzero stalk. A stalk over a point $y \in V$, consists of all $s : \pi^{-1}(y)  \to \cM_{{\bf n},E}$, such that  
$s(a x) = \vartheta_{\bf n}(a) s(x),$ for all $a \in F^\times$ and all $x \in \pi^{-1}(y).$ This is possible if and only if $\vartheta_{\bf n}(a) = 1$ for 
all $a \in F^\times \cap \U_{F,f}(\m) = U_F^1(\m)$. After passing to $\bar\Q$ via an $\iota$, the rest is basically the same as proof of Lem.\,\ref{lem:purity}. 
\end{proof}

\medskip

Analogously, one may construct sheaves on the other spaces in \eqref{eqn:loc-symm-space}; one may pull-back or push-forward the so constructed sheaves, and explicate the mutual relations between them.

\bigskip
\subsection{Cohomology of $\GL_1$ and algebraic Hecke characters of $F$ with coefficients in $E$}
Given an integral ideal $\m$, and an infinity type ${\bf n} \in \Z[\Hom(F,E)]$ satisfying the purity conditions of  Lem.\,\ref{lem:purity-sheaf}, 
a fundamental object of interest is 
$$
H^\bullet(S^G_\m, \, \tM_{{\bf n},E}), 
$$
the sheaf-cohomology of a locally symmetric space for $\GL_1/F$   
with coefficients in the sheaf $\tM_{{\bf n},E}.$ One may also consider 
$H^\bullet_c(S^G_\m, \, \tM_{{\bf n},E}),$
which is the cohomology with compact supports. 
It will take us way too much of an effort without substantial dividends for the purposes of this article to develop this systematically. 
Any reader who might feel this is brewing heavy weather is recommended to look into Harder \cite[\S\,2.5]{harder} on the cohomology of tori, where the torus at hand is the diagonal torus in $\GL_2$, which, up to the easily controlled centre, is basically the above kind of cohomology group. This context naturally arises in the theory of Eisenstein cohomology where the cohomology of the Levi subgroups (in particular, tori) play a crucial role. The reader is referred to Harder's 
forthcoming book \cite{harder-book} where this point of view is treated in great detail. 

\medskip
Let us suffice here to note that an algebraic Hecke character $\chi$ of $F$ modulo $\m$, 
taking values in the algebraic closure $\bar\Q$ of $\Q$ inside $\C$, and with infinity type ${\bf n}$, 
contributes to $H^0(S^G_\m, \, \tM_{{\bf n}, \bar\Q})$, and by duality also to the top-degree cohomology with compact supports: 
$H^{d_F}_c(S^G_\m, \, \tM_{{\bf n}, \bar\Q}).$ For arithmetic applications, it is better to work with an arbitrary coefficient field $E$, from which 
one can then map into $\bar\Q$ or $\C$; this necessitates clarifying the meaning of an algebraic Hecke character of $F$ with 
coefficients in $E$. 

\medskip

The definition below follows Deligne \cite[\S 5]{deligne-sga}; see also Schappacher \cite[Chap.\,0]{schappacher}. 
Whereas in \cite{deligne-sga} and  \cite{schappacher} $E$ is any number field with no relation to $F$, we will continue with our 
restriction that $E$ is a finite Galois extension of $\Q$ that contains a copy of $F.$ Given an integral ideal $\m$ of $F$, 
and an infinity type ${\bf n} \in \Z[\Hom(F,E)],$
define an {\it algebraic Hecke character of $F$ with values in $E$ of modulus $\m$ and infinity type ${\bf n}$} as a homomorphism
$$
\chi: \ \J_F(\m) \ \to \ E^\times
$$ 
such that for any principal ideal $(a) \in \P^+_F(\m)$ (recall: $a \gg 0$ and $a \equiv 1 \pmod{\m}$) one has 
$$
\chi((a)) \ = \ \prod_{\tau \in \Hom(F,E)} \tau(a)^{n_\tau}.
$$
Note the similarity to Hecke's definition of a {\it Gr\"o\ss encharakter} mod $\m$, but with values in $E^\times$ instead of $\C^\times$, and 
while asking for the character at infinity to be of algebraic type. 
To see the relation with our previous treatment of algebraic Hecke characters, following \cite[\S0.5]{schappacher}, `localise' over any place 
of $E$. We will work with archimedean places, but note that as in {\it loc.\,cit.\,}we may also go through the discussion below for any finite place which would give
us $\ell$-adic versions of algebraic Hecke characters. 

\medskip

Let $\chi$ be an algebraic Hecke character of $F$ with values in $E$ of modulus $\m = \prod_{\p|\m} \p^{m_\p}$ and infinity type ${\bf n}.$ Then there is unique continuous homomorphism
\begin{equation}
\label{eqn:chi-A}
\chi_\A : \I_F \ \to \ E^\times,
\end{equation}
with the discrete topology on $E^\times$, such that 
\begin{enumerate}
\smallskip
\item $\chi_\A^{-1}(1) = \prod_{v \in S_\infty}F_v^\times \prod_{\p | \m} (1+\p^{m_\p}\O_p) \prod_{\p \not\,| \, \m} \O_\p^\times,$ which is open in $\I_F.$  
\smallskip
\item for $\p \! \not| \, \m$, the local component $\chi_\p$ of $\chi_\A$ satisfies: $\chi_\p(\varpi_\p) = \chi(\p).$ 
\smallskip
\item $\chi_\A|_{F^\times} = \vartheta_{\bf n};$ (recall: $\vartheta_{\bf n}(a) = \prod_{\tau : F \to E} \tau(a)^{n_\tau}$).  
\end{enumerate}
By weak approximation, $\I_F = F^\times \cdot (\prod_{v \in S_\infty}F_v^\times \prod_{\p | \m} (1+\p^{m_\p}\O_p) \prod_{\p \not\,| \, \m} F_\p^\times)$, 
giving us the uniqueness of $\chi_\A.$ Let us note {\it en passant} that such a $\chi_\A$ is an element of $H^0(S^G_\m, \tM_{{\bf n},E}).$ 
Fix an embedding $\iota : E \hookrightarrow \C$, and suppose it corresponds to the archimedean 
place $v_\iota$ of $E$. The character $\vartheta_{\bf n}$ is an element of $\Hom_{\Q-{\rm alg}}(\Res_{F/\Q}(\GL_1/F), \Res_{E/\Q}(\GL_1/E))$, and taking 
the $\A$-points, one gets a continuous homomorphism $\vartheta_{{\bf n}, \A} : \I_F \to \I_E.$ 
Composing with the projection $\I_E^\times \to E_{v_\iota}^\times \simeq \C^\times$ gives:  
\begin{equation}
\label{eqn:theta-iota}
\vartheta_{{\bf n}, \iota} : \I_F \to \C^\times.
\end{equation}
Now, define a homomorphism $\I_F \to \C^\times$ as: 
\begin{equation}
\label{eqn:iota-chi}
{}^\iota\chi \ := \ (\iota \circ \chi_\A) \cdot \vartheta_{{\bf n},\iota}^{-1},
\end{equation}
Note that $\iota$ induces a bijection 
$\iota_* : \Hom(F,E) \to \Hom(F,\C)$ as $\iota_*(\tau) = \iota \circ \tau.$ The standing hypothesis on $E$ being Galois and containing a copy of $F$
is used here.  In particular, there is a mapping of infinity types: 
$\Z[\Hom(F,E)] \to \Z[\Hom(F,\C)]$ as ${\bf n} \mapsto {}^\iota{\bf n};$
for $\eta \in \Hom(E,\C)$, the $\eta$-component of ${}^\iota{\bf n}$ is given by:
${}^\iota{\bf n}_\eta = n_{\iota_*^{-1}\eta}.$
Hence, $\iota \circ \vartheta_{\bf n} = \vartheta_{{}^\iota{\bf n}}$, as homomorphisms from $F^\times$ to $\C^\times$ 
since 
$$
(\iota \circ \vartheta_{\bf n})(a)
 \ = \ \iota (\prod_{\tau \in \Hom(F,E)} \tau(a)^{n_\tau}) 
\ = \ \prod_{\tau \in \Hom(F,E)} (\iota_*\tau(a))^{n_\tau} 
\ = \ \prod_{\eta \in \Hom(F,\C)} (\eta(a))^{n_{\iota_*^{-1}\eta}}.
$$
The properties of ${}^\iota\chi$ are summarised in the following proposition the proof of which is now a routine check. 

\begin{prop}
\label{prop:chi-E-to-C}
Let $\chi$ be an algebraic Hecke character of $F$ with values in $E$ of modulus $\m$ and infinity type ${\bf n}.$ For any embedding 
$\iota : E \to \C$, the homomorphism ${}^\iota\chi$ satisfies the following: 
\begin{enumerate}
\item ${}^\iota\chi : \I_F/F^\times \to \C^\times$ is a continuous homomorphism; 
\item For all $\p \notin S_\infty$, ${}^\iota\chi_\p = \iota \circ \chi_\p$; recall that $\chi_\p$ is the local character of $\chi_\A;$ 
\item ${}^\iota\chi_\infty = \vartheta_{{\bf n},\iota}^{-1}|_{F_\infty^\times}$ is completely determined by its values on 
$F^\times$ embedded diagonally in $F_\infty^\times,$ on which we have
$\vartheta_{{\bf n},\iota}^{-1}|_{F^\times} = \iota \circ \vartheta_{\bf n}^{-1} = \vartheta_{-{}^\iota{\bf n}};$ 
\end{enumerate}
i.e, ${}^\iota\chi$ is a Hecke character of $F$ of modulus $\m$ and infinity type $-{}^\iota{\bf n}.$ 

\smallskip
Furthermore, for $k \in \Z$, the Tate-twist ${}^\iota\chi(k) := {}^\iota\chi \otimes |\!|\ |\!|^k$, has infinity type $-{}^\iota{\bf n}+k.$
\end{prop}

We may say that ${}^\iota\chi$ is an automorphic representation of $\GL_1/F$ of cohomological type. We are interested in the special values of the 
Hecke $L$-function attached to ${}^\iota\chi$.

\bigskip
\section{\bf Ratios of critical values of Hecke $L$-functions}
\label{sec:harder-hecke}

Let $\chi$ be an algebraic Hecke character of $F$ with values in $E$ of modulus $\m$ and infinity type ${\bf n} = \sum_{\tau : F \to E} n_\tau \tau.$  
For any embedding $\iota : E \to \C$, consider the algebraic Hecke character ${}^\iota\chi$ as in Prop.\,\ref{prop:chi-E-to-C}. The rest 
of this article concerns the special values of the $\C$-valued $L$-function:
$$
L(s,\iota, \chi) \ := \ L(s,  {}^\iota\chi),
$$
the right hand side is the usual Hecke $L$-function; see \ref{sec:defn-l-fn}. 
It is sometimes convenient to consider the $E \otimes \C$-valued $L$-function: 
$$
\bL(s, \chi) \ := \ \{L(s,\iota, \chi)\}_{\iota: E \to \C}, 
$$
where $E \otimes \C$ is identified with $\prod_{\iota: E \to \C} \C.$ From Prop.\,\ref{prop:critical} and Lem.\,\ref{lem:ell-ind-iota} below 
it follows that the 
critical set of $L(s,  {}^\iota\chi)$ is independent of $\iota$; one may therefore allude to the critical set for $\bL(s, \chi).$ 
Furthermore, we are interested in the rationality results for the ratios of successive 
critical $L$-values that arise naturally in the theory of Eisenstein cohomology as in Harder \cite{harder}. 
Henceforth, we assume that $F$ is a totally imaginary number field, in which case it follows from Prop.\,\ref{prop:critical} 
that if the width $\ell$ of ${}^\iota\chi$ satisfies $\ell \geq 2$, then we are guaranteed the existence of two consecutive integers 
that are critical for $\bL(s, \chi).$

\bigskip
\subsection{A result of Harder on ratios of critical $L$-values}

The relative discriminant $\delta_{F/K}$ of any extension $F/K$ of number fields is defined as: suppose $\Hom(F, \bar{K}) = \{\sigma_1,\dots, \sigma_r\}$, and 
$\{\omega_1,\dots,\omega_r\}$ is a $K$-basis of $F$, then define $\delta_{F/K} \ := \ \det([\sigma_i(\omega_j)])^2$; 
taking the square makes it invariant under $\Gal(\bar{K}/K)$, and hence $\delta_{F/K}$ is in $K^\times$, and as an element of $K^\times/K^\times{}^2$ it is 
independent of the choice of the basis $\{\omega_j\}$. The absolute discriminant of $F$ is $\delta_{F/\Q}.$ 
Harder proves in \cite[Cor.\,(4.2.2)]{harder} that if $m$ and $m+1$ are critical for $\bL(s, \chi),$ then the number
\begin{equation}
\label{eqn:harder-1}
C({}^\iota\chi) \ := \ |\delta_{F/\Q}|^{1/2} \frac{L(m, {}^\iota\chi)}{L(m+1, {}^\iota\chi)}
\end{equation}
is in $\bar\Q$, and for any $\varsigma \in \Gal(\bar\Q/\Q)$, one has 
\begin{equation}
\label{eqn:harder-2}
\varsigma(C({}^\iota\chi)) \ = \ C({}^{\varsigma\circ\iota}\chi). 
\end{equation}
Let us note that we work with the completed $L$-function, and not with just the finite part of the $L$-function as in {\it loc.\,cit.} In particular, the power of 
$\bpi$ in \cite[Cor.\,(4.2.2)]{harder} is subsumed by the ratio of archimedean local-factors. Also, the result is stated therein for $m=-1$; for a general critical integer $m$ (with $m+1$ also critical) follows by taking Tate-twists of $\chi.$

\medskip

Note that \eqref{eqn:harder-1} is asserting that the ratio of $L$-values therein is algebraic, since $|\delta_{F/\Q}|^{1/2}$ is algebraic; but if we only look 
at the ratios of $L$-values then the reciprocity law becomes obscured. At any rate, the algebraicity result as stated in \eqref{eqn:harder-1}
together with the reciprocity law in \eqref{eqn:harder-2} gives that: 
\begin{equation}
\label{eqn:harder-3}
|\delta_{F/\Q}|^{1/2} \frac{L(m, {}^\iota\chi)}{L(m+1, {}^\iota\chi)} \ \in \iota(E).
\end{equation}
We contend that whereas just the algebraicity result in \eqref{eqn:harder-1} is correct, but restated as the stronger statement in \eqref{eqn:harder-3} via the reciprocity law it is not correct as it stands by producing an example when $F$ is a totally imaginary field that is not of CM type. 
The reason turns out to be, and this is the {\it raison d'\^etre} for 
writing this article, that the reciprocity law in \eqref{eqn:harder-2} needs to be modified by a sign; this sign is rather complicated, and turns out to be trivial when
$F$ is a CM field, and can be nontrivial for a general totally imaginary field. 
We will first discuss this example in Sect.\,\ref{sec:counterexample} that says that \eqref{eqn:harder-3} is not stable under base-change making it incorrect as it stands in general. In Sect.\,\ref{sec:correct-theorem}, we will see how to rectify the reciprocity law in \eqref{eqn:harder-2} using, interestingly enough, some other passages in \cite{harder} itself.

\bigskip
\subsection{An example: stability under base-change} 
\label{sec:counterexample}

Consider the fields in Example\,\ref{exam:hida}: $F_0 = \Q$, and put $F_1 = \Q(\i)$ and $F = \Q(\i, \sqrt{4+\i}).$ To emphasize, this $F$ is totally imaginary but not CM, and $F_1$ is a CM subfield. Changing notations mildly to make it more suggestive we denote:  
$$
\xymatrix{
F = \Q(\i, \sqrt{4+\i}) \ar@{-}[d] \\
F^\cm = \Q(\i) \ar@{-}[d] \\
F^+ = \Q
}
$$ 
Let us compute $\delta_{F/\Q}.$ First of all, 
$$
\delta_{F/F^\cm} \ = \ \delta_{\Q(\i, \sqrt{4+\i})/\Q(\i)} \ = \ 4(4+\i), 
$$ 
computed with respect to the basis $\{1, \sqrt{4+\i}\}$; hence: 
$$
N_{F^\cm/\Q}(\delta_{F/F^\cm}) \ = \ N_{\Q(\i)/\Q}(4(4+\i)) \ = \ 16\cdot 17.
$$
Also, using the basis $\{1, \i\}$ for $\Q(\i)$, we have:
$$
\delta_{F^\cm/\Q} = \delta_{\Q(\i)/\Q} = -4.
$$
Using the well-known formula for the discriminant of a tower of fields, we get:
$$
\delta_{F/\Q} \ = \ \delta_{F^\cm/\Q}^{[F:F^\cm]} \cdot N_{F^\cm/\Q}(\delta_{F/F^\cm}) \ = \ (-4)^2 \cdot 16 \cdot 17 \ =  \
2^8 \cdot 17.
$$

\bigskip

Let $\psi$ be an algebraic Hecke character of $F^\cm = \Q(\i)$; as a continuous homomorphism $\psi : \A_{\Q(\i)}^\times / \Q(\i)^\times \to \C^\times$, 
with the character at infinity, $\psi_\infty : \C^\times \to \C^\times,$ being of the form $\psi_\infty(z) = z^a \bar{z}^b,$ for $a , b \in \Z$. We may take (without loss of much generality) that $a \geq b$. The critical set for $L(s, \psi)$ is given by: 
$$
\Crit(L(s, \psi)) \ = \ \{1-a, \, 2-a, \, \dots, \, -b\}.
$$
Let us further assume that $a -b \geq 2$ so that there are at least two critical points; let $m, m+1 \in \Crit(L(s, \psi)).$ Then, 
\eqref{eqn:harder-3} will say that 
$$
|\delta_{F^\cm/\Q}|^{1/2} \ \frac{L(m,\psi)}{L(m+1, \psi)} \ \in \Q(\psi), 
$$
where $\Q(\psi)$ is the number field generated by the values of the finite part $\psi_f$ of $\psi.$ (We could have set up the context so that $\psi$ takes values in a coefficient field $E$, and then used an 
$\iota : E \to \C$ and to assert that the left hand side, for $L(s, \iota, \psi)$ in place of $L(s, \psi),$ lies in $\iota(E).$) Since 
$|\delta_{F^\cm/\Q}|^{1/2} = |-4|^{1/2} \in \Q^\times$ we have: 
\begin{equation}
\label{eqn:psi}
\frac{L(m,\psi)}{L(m+1, \psi)} \ \in \ \Q(\psi).
\end{equation}
Suppose $\omega$ is a quadratic Dirichlet character of $F^\cm$, then $\omega_\infty$ is trivial, and \eqref{eqn:psi} applied to $\psi\omega$ will give: 
\begin{equation}
\label{eqn:psi-omega}
\frac{L(m,\psi \omega)}{L(m+1, \psi \omega)} \ \in \ \Q(\psi\omega) \ = \ \Q(\psi).
\end{equation}

\bigskip

Now let $\chi$ be the base-change of $\psi$ to give a Hecke character of $F$:
$$
\chi \ := \ \psi \circ N_{F/F^\cm}. 
$$
Note that $F_\infty = \C \times \C$, $F^\cm_\infty = \C$, and $N_{F_\infty/F^\cm_\infty}(z_1,z_2) = z_1z_2$. Hence: 
$\chi_\infty(z_1, z_2) = (z_1z_2)^a \overline{(z_1z_2)}^b = z_1^a z_2^a \bar{z_1}^b \bar{z_2}^b,$ which is to say 
$\chi_\infty = \psi_\infty \otimes \psi_\infty$, whence $L_\infty(s, \chi) = L_\infty(s, \psi)^2$, from which we deduce:
$$
\Crit(L(s, \chi)) \ = \ \Crit(L(s, \psi)). 
$$
Also, a basic property of base-change is that:
\begin{equation}
\label{eqn:prod}
L(s, \chi) \ = \ L(s, \psi) L(s, \psi \omega),
\end{equation}
where $\omega = \omega_{F/F^\cm}$ is the quadratic character of $F^\cm$ attached by class field theory to the quadratic extension $F/F^\cm.$ 
From \eqref{eqn:prod}, \eqref{eqn:psi}, and \eqref{eqn:psi-omega} we deduce: 
\begin{equation}
\label{eqn:chi-1}
\frac{L(m,\chi)}{L(m+1, \chi)} \ \in \ \Q(\psi) = \Q(\chi).
\end{equation}
But, on the other hand, one may directly apply \eqref{eqn:harder-3} to $L(s, \chi)$ to get: 
$$
|\delta_{F/\Q}|^{1/2} \ \frac{L(m,\chi)}{L(m+1, \chi)} \ \in \Q(\chi).
$$
For the discriminant term: $|\delta_{F/\Q}|^{1/2} = (2^8 \cdot 17)^{1/2} = \sqrt{17} \pmod{\Q^\times}.$ Hence: 
\begin{equation}
\label{eqn:chi-2}
\sqrt{17} \, \frac{L(m,\chi)}{L(m+1, \chi)} \ \in \ \Q(\chi).
\end{equation}
Of course, \eqref{eqn:chi-1} and \eqref{eqn:chi-2} are contradictory, because in general $\sqrt{17} \notin \Q(\chi).$ We may say that 
that the formal statement in \eqref{eqn:harder-3} is not stable under base-change. The reader can readily appreciate that the above example can be generalized.

\bigskip
\subsection{Variations on a result of Harder} 
\label{sec:correct-theorem}

In this subsection we state the main theorem of this article; see Theorem \ref{thm:main}. Before doing so, we need to introduce notions and notations that lead to a delicate signature that appears in that theorem. For the rest of this subsection we will let $F$ be a totally imaginary number field, and $E$ a number field that is Galois over $\Q$ and contains a copy of $F$. Let 
${\bf n} = \sum_\tau n_\tau \tau  \in \Z[\Hom(F, E)]$ be an infinity type satisfying the purity condition of Lem.\,\ref{lem:purity-sheaf}, 
with purity weight ${\sf w}.$
Suppose 
$\chi$ is an algebraic Hecke character of $F$ with values in $E$ of infinity type ${\bf n}.$ Take any embedding $\iota : E \to \C,$ and 
let ${}^\iota\chi$ be the character as in Prop.\,\ref{prop:chi-E-to-C}.  

\medskip
Recall from Prop.\,\ref{prop:critical} that the width of ${\bf n},$ for a given $\iota,$ is defined as:  
$$
\ell({}^\iota\chi) \ = \ \ell({}^\iota{\bf n}) \ := \ \min_{\tau : F \to E} \{|n_{\iota\circ\tau} - n_{\overline{\iota\circ\tau}}|\}.
$$ 
The Hecke $L$-function $L(s, {}^\iota\chi)$ has $\ell({}^\iota\chi)$ many critical points. Note the following simple but important lemma: 

\begin{lemma}
\label{lem:ell-ind-iota}
$\ell({}^\iota{\bf n})$ is independent of $\iota$, and depends
only on the pure infinity type ${\bf n}.$ 
\end{lemma}

\begin{proof}
If $F$ has no CM subfield ($F_1 = F_0$ is totally real), then from Prop.\,\ref{prop:exist_alg_hecke_tot_imag}
is follows that $\ell({}^\iota{\bf n}) = 0,$ which proves the lemma; but this case is not interesting because we would like to have critical points. 
Assume henceforth that $F$ is totally imaginary field that contains a (maximal) 
CM subfield $F_1,$ in which case ${\bf n}$ is the base-change from a pure infinity type ${\bf m}$ over $F_1$. It is clear that 
$\ell({}^\iota{\bf n}) = \ell({}^\iota{\bf m}).$ Suppose $\Hom(F_1,E) = \{\tau_1, \tau_1', \dots, \tau_k, \tau_k'\}$, where $2k = [F_1:\Q]$, 
and $\tau_j$ and $\tau_j'$ have the same restriction to the maximal totally real subfield $F_0$ of $F_1$. Then it  follows from (the proof of) 
Lem.\,\ref{lem:CM-field-embeddings} that: 
$\ell({}^\iota{\bf m}) \ = \ \min_{j} \{|m_{\tau_j} - m_{\tau_j'}|\}.$   
\end{proof}
The above lemma justifies the notation $\ell = \ell({\bf n}).$

\medskip
\subsubsection{\bf A combinatorial lemma} 
\label{sec:comb-lemma}

 The following is a special case 
of a general `combinatorial lemma' that is at the heart of what makes the theory of Eisenstein cohomology work so well for studying the special values 
to automorphic $L$-functions.  

\begin{lemma}
\label{lem:comb-lemma}
The following three conditions are equivalent: 
\begin{enumerate}
\item $s = -1$ and $s = 0$ are critical for $L(s, {}^\iota\chi).$ 
\item $-\ell \leq {\sf w} \leq -4 + \ell.$ 
\item For each $\tau : F \to E$ and $\iota : E \to \C$, there exists $w_{\iota\circ\tau}$ in the Weyl group of $\GL(2)$ such that 
   \begin{enumerate}
   \item $l(w_{\iota\circ\tau}) + l(w_{\overline{\iota\circ\tau}}) = 1,$ and 
   \item $w_{\iota\circ\tau} \cdot \left(\begin{smallmatrix} n_{\iota\circ\tau} & \\ & 0 \end{smallmatrix}\right)$ is dominant. 
   \end{enumerate}
\end{enumerate}
\end{lemma}
The reader is referred to \cite[Sect.\,3.2]{raghuram-tot-imag} for a proof; to compare Statement (2) above to the corresponding statement in \cite[Lem.\,3.16]{raghuram-tot-imag}, note
that $N$, $\ell(\mu,\mu')$, and $a(\mu, \mu')$ of \cite{raghuram-tot-imag} take the values 
$2$, $\ell$, and ${\sf w}/2,$ respectively in this article; Statement (3) and its consequences need to be elaborated upon. 
Given $n_1, n_2 \in \Z$, by $\left(\begin{smallmatrix} n_1 & \\ & n_2 \end{smallmatrix}\right)$ one means the algebraic character of the diagonal torus of 
$\GL(2)$ given by ${\rm diag}(t_1,t_2) \mapsto t_1^{n_1} t_2^{n_2}.$ Such a character is said to be dominant if $n_1 \geq n_2$ (where the choice of 
the Borel subgroup $B$ is taken to be the upper-triangular matrices in $\GL(2)$). 
The Weyl group $W$ of $\GL(2)$, which has order $2$, acts on this torus and hence 
on its characters. The twisted action of the nontrival element $w_0 \in W$ on such a character 
is explicitly given by: 
$w_0 \cdot \left(\begin{smallmatrix} n_1 & \\ & n_2 \end{smallmatrix}\right) =  
\left(\begin{smallmatrix} n_2-1 & \\ & n_1+1 \end{smallmatrix}\right).$
Of course the length $l(w)=1$ if $w = w_0$, and $l(w) = 0$ if $w$ is the trivial element. Now the statement in (3) makes sense. 
Furthermore, in each pair 
$\{\iota\circ\tau, \, \overline{\iota\circ\tau}\}$ of conjugate embeddings $\Sigma_F$, by (a), exactly one of the Weyl group elements is nontrivial and so having length $1$, and 
the other is trivial; this defines a CM-type for $F$ (recall that by a CM-type $\Phi$ for $F$ one means $\Phi \subset \Sigma_F$ such that 
$\Sigma_F = \Phi \amalg \bar{\Phi}$):  
\begin{equation}
\label{eqn:CM-type-n-iota}
\Phi({\bf n}, \iota) \ := \ \left\{ \eta \in \Sigma_F \ :  \ l(w_\eta) = 1\right\}, 
\end{equation}
and a bijection $\beta_{{\bf n}, \iota} : \Phi({\bf n}, \iota) \to S_\infty$ with the set of archimedean places 
of $F.$ Statement (3) also implies: 
\begin{equation}
\label{eqn:CM-type-n-iota-consequence}
\eta \in \Phi({\bf n}, \iota) \implies n_\eta \leq -2, \quad {\rm and} \quad \eta \notin \Phi({\bf n}, \iota) \implies n_\eta \geq 0. 
\end{equation}  
For each $\eta \in \Sigma_F$ we have ($n_\eta \leq -2$ and $n_{\bar\eta} \geq 0$) or ($n_\eta \geq 0$ and $n_{\bar\eta} \leq -2$).  

\medskip
Let us digress for a moment: what condition (3) really says is that the automorphic representation 
$\aInd_{B(\A_F)}^{\GL_2(\A_F)}({}^\iota\chi, 1\!\!1),$ that is algebraically and parabolically induced from the Borel subgroup, 
contributes to the cohomology in degree $[F:\Q]/2$ of the 
Borel--Serre boundary $\partial_BS^G$ of a locally symmetric space $S^G$ for $G = \GL(2)/F$. This is the starting point of using Eisenstein cohomology for $\GL(2)/F$ 
to deduce results on the special values of $L$-functions attached to algebraic Hecke characters that was pioneered by Harder in \cite{harder}. 
In that paper, Harder proves, under the conditions of Lem.\,\ref{lem:comb-lemma}, a rationality result for the ratio of critical values 
$L(-1, {}^\iota\chi)/L(0, {}^\iota\chi),$ from which an analogous rationality result follows for the ratios of other successive pair of critical values by allowing ourselves 
Tate-twists of $\chi$; Statement (2) bounds the possible twists we can make, and it is a striking aspect of the lemma that one so obtains all successive pairs 
of critical values, {\it no more and no less!} Let me note that in \cite{harder} such a lemma and its consequences are implicit; the explicit version of the above lemma and its philosophical content came out in my work with Harder \cite{harder-raghuram-book}; that was further developed in my sequel \cite{raghuram-tot-imag}--especially, see paragraph 5.3.2.2 in \cite{raghuram-tot-imag}.

\bigskip
\subsubsection{\bf A tale of two signatures} 
\label{sec:two-signatures}
Fix an ordering on the set $S_\infty$ of archimedean places 
of $F$; say $S_\infty = \{w_1,\dots,w_{r}\},$ where $r = d_F/2 = [F:\Q]/2.$  
For each $\varsigma \in \Gal(\bar\Q/\Q)$ define a permutation $\pi_{{\bf n}, \iota}(\varsigma)$ of $S_\infty$ by the diagram:
$$
\xymatrix{
\Phi({\bf n}, \iota) \ar[rr]^{\beta_{{\bf n}, \iota}} \ar[d]_{\varsigma \circ -} & & S_\infty \ar[d]^{\pi_{{\bf n}, \iota}(\varsigma)} \\ 
\Phi({\bf n}, \varsigma\circ\iota) \ar[rr]^{\beta_{{\bf n}, \varsigma\circ\iota}} & & S_\infty, 
}
$$
where the left vertical arrow needs an explanation: 
the $\iota$-independence of (1) and/or (2) of Lem.\,\ref{lem:comb-lemma} translates to $\iota$-independence of (3) also, whence  
if $\iota \circ \tau \in \Phi({\bf n}, \iota)$, then $\varsigma  \circ (\iota \circ \tau) \in \Phi({\bf n}, \varsigma \circ \iota).$ Put in other words, 
in general, a Galois conjugate of a CM-type of a totally imaginary field need not be a CM-type; however, since ${\bf n}$ is a base-change of an infinity-type ${\bf m}$ from $F_1$ (the key ingredient that went into $\iota$-independence) we see that the CM-type $\Phi({\bf n}, \iota)$ for $F$ is deduced from the CM-type 
$\Phi({\bf m}, \iota)$ for $F_1$ under the canonical map $\Sigma_F \to \Sigma_{F_1};$ and for a CM-field such as $F_1$, a Galois conjugate of a CM-type is 
a CM-type; hence we deduce that $(\varsigma \circ -)(\Phi({\bf n}, \iota)) = \Phi({\bf n}, \varsigma\circ\iota)$ is a CM-type. The sign of the permutation 
$\pi_{{\bf n}, \iota}(\varsigma)$ defines 
the signature: 
\begin{equation}
\label{eqn:signature}
\varepsilon_{{\bf n}, \iota}(\varsigma) \ := \ {\rm sgn}(\pi_{{\bf n}, \iota}(\varsigma)). 
\end{equation}

\medskip
Let us digress for another moment: the key idea in the theory of Eisenstein cohomology is to see the map induced in cohomology by 
the standard intertwining operator: 
$$
T_{\rm st}(s)|_{s = -1} \ : \ 
\aInd_{B(\A_F)}^{\GL_2(\A_F)}({}^\iota\chi, 1\!\!1) \ \longrightarrow \ 
\aInd_{B(\A_F)}^{\GL_2(\A_F)}(1\!\!1(1), {}^\iota\chi(-1))
$$
is realised in terms of the restriction map from total cohomology of $S^G$ to the cohomology of the isotypic component cut out by these induced representations in the cohomology of the Borel--Serre boundary $\partial_BS^G$. (See \cite[pp.\,82--83]{harder-raghuram-book}.)
The infinity type of the inducing representation ${}^\iota\chi \otimes 1\!\!1$ in the domain of $T_{\rm st}(s)|_{s = -1}$ corresponds to the weight 
${}^\iota\underline{\bf n} := \left(\begin{smallmatrix} {}^\iota{\bf n} & \\ & 0 \end{smallmatrix}\right)$; similarly  
the infinity type of the inducing representation $1\!\!1(1) \otimes {}^\iota\chi(-1)$ for the codmain is the weight 
${}^\iota\tilde{\underline{\bf n}} := \left(\begin{smallmatrix} -1 & \\ & {}^\iota{\bf n}+1 \end{smallmatrix}\right)$. 

\medskip
Complementary to $\Phi({\bf n}, \iota)$, we define the CM-type 
\begin{equation}
\label{eqn:CM-type-n-iota-tilde}
\Phi(\tilde{\bf n}, \iota) \ := \ \left\{\eta \in \Sigma_F \ :  \ l(w_\eta) = 0 \right\} \ = \ \left\{\eta \in \Sigma_F \ :  n_\eta \geq 0\right\}.
\end{equation}
Note that $\Phi({\bf n}, \iota) = \overline\Phi(\tilde{\bf n}, \iota),$ and the bijection 
$\beta_{\tilde{\bf n}, \iota} : \Phi(\tilde{\bf n}, \iota) \to S_\infty$ is the complex-conjugate of $\beta_{{\bf n}, \iota}.$
For each $\varsigma \in \Gal(\bar\Q/\Q)$ define a permutation $\pi_{\tilde{\bf n}, \iota}(\varsigma)$ of $S_\infty$ by the diagram:
$$
\xymatrix{
\Phi(\tilde{\bf n}, \iota) \ar[rr]^{\beta_{\tilde{\bf n}, \iota}} \ar[d]_{\varsigma \circ -} & & S_\infty \ar[d]^{\pi_{\tilde{\bf n}, \iota}(\varsigma)} \\ 
\Phi(\tilde{\bf n}, \varsigma\circ\iota) \ar[rr]^{\beta_{\tilde{\bf n}, \varsigma\circ\iota}} & & S_\infty, 
}
$$
the sign of which defines the signature 
\begin{equation}
\label{eqn:signature-tilde}
\varepsilon_{\tilde{\bf n}, \iota}(\varsigma) \ := \ {\rm sgn}(\pi_{\tilde{\bf n}, \iota}(\varsigma)). 
\end{equation}

\bigskip
\subsubsection{\bf The main theorem on $L$-values} 

We are now ready to state the main theorem which is a variation of a result due to Harder:

\begin{thm}
\label{thm:main}
Let $F$ be a totally imaginary number field, and $E$ a number field that is Galois over $\Q$ and containing a copy of $F$. Let 
${\bf n} = \sum_\tau n_\tau \tau  \in \Z[\Hom(F, E)]$ be an infinity type satisfying the purity condition of Lem.\,\ref{lem:purity-sheaf}, 
with purity weight ${\sf w}.$
Suppose 
$\chi$ is an algebraic Hecke character of $F$ with values in $E$ of infinity type ${\bf n}.$ Take any embedding $\iota : E \to \C,$ and 
let ${}^\iota\chi$ be the character as in Prop.\,\ref{prop:chi-E-to-C}.  
Assume the condition on ${\bf n}$ from  Lem.\,\ref{lem:comb-lemma}: 
$$-\ell \leq {\sf w} \leq -4 + \ell.$$  
Hence $\ell \geq 2$ and ${\bf n}$ is the base-change to $F$ from an infinity-type of a maximal CM-subfield of $F$. 
Suppose $m, m+1 \in {\rm Crit}(\bL(s, \chi));$ see Prop.\,\ref{prop:critical}. 
Then: 
\begin{equation}
\label{eqn:harder-variation-1}
|\delta_{F/\Q}|^{1/2} \frac{L(m, {}^\iota\chi)}{L(m+1, {}^\iota\chi)} \ \in \iota(E), 
\end{equation}
and, furthermore, for every $\varsigma \in \Gal(\bar\Q/\Q)$ we have the reciprocity law: 
\begin{equation}
\label{eqn:harder-variation-2}
\varsigma \left(|\delta_{F/\Q}|^{1/2} \frac{L(m, {}^\iota\chi)}{L(m+1, {}^\iota\chi)}\right) \ = \ 
\varepsilon_{{\bf n}, \iota}(\varsigma) \cdot 
\varepsilon_{\tilde{\bf n}, \iota}(\varsigma) \cdot
|\delta_{F/\Q}|^{1/2} 
\frac{L(m, {}^{\varsigma\circ\iota}\chi)}{L(m+1, {}^{\varsigma\circ\iota}\chi)}.
\end{equation}
\end{thm}

\bigskip
\subsubsection{\bf Remarks on the proof}
\label{sec:trivial-sign-CM}
The above theorem is a special case of a more general theorem on Rankin--Selberg $L$-functions 
for $\GL(n) \times \GL(n')$ over a totally imaginary field; see \cite[Thm.\,5.16, (1), (2)]{raghuram-tot-imag}. 
When $F$ is a CM field, one has $\varsigma \circ \c \circ \eta = \c \circ \varsigma \circ \eta$ (see Lem.\,\ref{lem:CM-field-embeddings}), 
hence, 
$\pi_{{\bf n}, \iota}(\varsigma) = \pi_{\tilde{\bf n}, \iota}(\varsigma)$ an equality of the permutations of $S_\infty$; whence, 
$\varepsilon_{{\bf n}, \iota}(\varsigma) = \varepsilon_{\tilde{\bf n}, \iota}(\varsigma).$ In particular, 
the reciprocity law in \eqref{eqn:harder-variation-2} is the same as \eqref{eqn:harder-2}. 
As will be seen below this is no longer the case for a general total imaginary field; even though one has an equality of sets
$\c(\Phi({\bf n}, \varsigma\circ\iota)) = \varsigma(\c(\Phi({\bf n}, \iota))),$ the permutations $\pi_{{\bf n}, \iota}(\varsigma)$ and $\pi_{\tilde{\bf n}, \iota}(\varsigma)$
can be distinct, and the signature $\varepsilon_{{\bf n}, \iota}(\varsigma) \cdot \varepsilon_{\tilde{\bf n}, \iota}(\varsigma)$ can be nontrivial. (This is in essence the explanation of the counterexample in Sect.\,\ref{sec:counterexample}.)
Even though Thm.\,\ref{thm:main} and especially \eqref{eqn:harder-variation-2}, 
for a general totally imaginary field, are not stated like this 
in Harder\,\cite{harder}, its proof follows by putting together the proof of 
\cite[Cor.\,4.2.2]{harder} amplified by the discussion in \cite[Sect.\,2.4]{harder} of a rational system of generators for the unipotent cohomology of the coefficient system on $\GL(2).$ The signature $\varepsilon_{{\bf n}, \iota}(\varsigma)$ is described in \cite[(2.4.1)]{harder}.

\bigskip
\subsection{Compatibility with Deligne's Conjecture}
In this subsection we recall Deligne's celebrated conjecture on the special values of motivic $L$-functions, as applied to ratios of successive critical values, 
and also recall a result of Deligne on the ratios of motivic periods  
that gives a motivic explanation to the ratio of successive critical $L$-values for a Hecke $L$-function. 

\subsubsection{\bf Statement of Deligne's Conjecture}
We will freely use the notation of \cite{deligne-corvallis}; a motive $M$ over $\Q$ with coefficients in a field $E$ will be thought in terms of its Betti, de Rham, and $\ell$-adic realizations. For a critical motive $M$ we have its periods $c^\pm(M) \in (E \otimes \C)^\times$ as in {\it loc.\,cit.,} that are well-defined in $(E \otimes \C)^\times/E^\times.$ We begin with a relation between the two periods over a totally imaginary base field $F$; recall that $F_0$ is the largest totally real subfield of $F$; then $F_0$ admits at most one totally imaginary quadratic extension contained inside $F$; if $F$ does admit such a CM subfield, we denote it as $F_1$; and if there is no such extension inside $F$, then we put $F_1 = F_0.$ If $F_1$ is indeed a CM field, and suppose $F_1 = F_0(\sqrt{D})$ for a totally negative $D \in F_0$, then define 
$$
\Delta_{F_1} \ := \ \sqrt{N_{F_0/\Q}(D)}, \quad \Delta_F \ := \ \Delta_{F_1}^{[F:F_1]}.
$$ 
If $F_1 = F_0$ is totally real, then define
$$
\Delta_{F_1} \ := \ 1, \quad \Delta_F \ := \ \Delta_{F_1}^{[F:F_1]} = 1. 
$$ 
The following proposition is proved in \cite{deligne-raghuram}: 

\medskip
\begin{prop}
\label{prop:deligne}
Let $M_0$ be a pure motive of rank $1$ over a totally imaginary number field $F$ with coefficients in a number field $E$. Put $M = \Res_{F/\Q}(M_0),$ and suppose that $M$ has no middle Hodge type. Let $c^\pm(M)$ be the periods defined in \cite{deligne-corvallis}. 
Then 
$$
\frac{c^+(M)}{c^-(M)} \ = \  1 \otimes \Delta_F, \quad \mbox{in $(E \otimes \C)^\times/ E^\times$.}
$$
\end{prop}

Let us also note that $1 \otimes \Delta_F$ is $\pm 1$ in each component of $(E \otimes \C)^\times/ E^\times$, since its square is trivial.  
Based on Prop.\,\ref{prop:deligne}, Deligne's conjecture \cite{deligne-corvallis} for the ratios of successive critical values of the completed $L$-function  of a rank-one motive may be stated as:  

\medskip
\begin{con}[Deligne]
\label{con:deligne}
Let $M_0$ be a pure motive of rank $1$ over a totally imaginary $F$ with coefficients in $E$. Put $M = \Res_{F/\Q}(M_0),$ and suppose that 
$M$ has no middle Hodge type. For $\iota : E \to \C$ let $L(s, \iota, M)$ denote the completed $L$-function attached to $(M, \iota).$ Put 
$L(s,M) =  \{L(s, \iota, M)\}_{\iota : E \to \C}$ for the array of $L$-functions taking values in $E \otimes \C.$ Suppose $m$ and $m+1$ are critical 
integers for $L(s,M)$, and assuming that $L(m+1,M) \neq 0$, we have
$$
\frac{L(m,M)}{L(m+1,M)} \ = \  1\otimes \i^{d_F/2} \Delta_F, 
$$
where the equality is in $(E \otimes \C)^\times/ E^\times.$
\end{con}

A word of explanation is in order, since, in \cite{deligne-corvallis}, Deligne formulated his conjecture for critical values of $L_f(s,M)$--the finite-part of the $L$-function attached to $M;$ which  takes the shape: 
$$
\frac{L_f(m, M)}{L_f(m+1, M)} \ = \ (1 \otimes (2 \bpi \i)^{-d^\pm(M)}) \, \frac{c^\pm(M)}{c^\mp(M)}, \quad \mbox{in $E \otimes \C$}.
$$
The assumption on $M$ that there is middle Hodge type gives $d^{\pm}(M) = d(M)/2 = d_F/2.$ 
Knowing the $L$-factor at infinity we have: 
$$
L_\infty(m, M)/L_\infty(m+1, M) = 1 \otimes (2\bpi)^{d_F/2}.
$$ 
For the completed $L$-function, we get:
\begin{equation}
\label{eqn:deligne-con-ratios}
\frac{L(m, M)}{L(m+1, M)} \ = \ (1 \otimes \i^{d_F/2}) \, \frac{c^\pm(M)}{c^\mp(M)}. 
\end{equation}
It is clear now that \eqref{eqn:deligne-con-ratios} and Prop.\,\ref{prop:deligne} give Conj.\,\ref{con:deligne}. 

\medskip
Using the conjectural correspondence between an algebraic Hecke character of $F$ with values in $E$, and a rank-one motive over $F$ with coefficients
in $E$, Conj.\,\ref{con:deligne} may be restated as the following:

\begin{con}[Deligne's Conjecture for ratios of critical values of Hecke $L$-functions]
\label{con:deligne-hecke}
Let the notations and hypothesis be as in Thm.\,\ref{thm:main}. 
Then: 
\begin{equation}
\label{eqn:harder-deligne-1}
\i^{d_F/2} \Delta_F \, \frac{L(m, {}^\iota\chi)}{L(m+1, {}^\iota\chi)} \ \in \iota(E), 
\end{equation}
and, furthermore, for every $\varsigma \in \Gal(\bar\Q/\Q)$ we have the reciprocity law: 
\begin{equation}
\label{eqn:harder-deligne-2}
\varsigma \left(\i^{d_F/2} \Delta_F \, \frac{L(m, {}^\iota\chi)}{L(m+1, {}^\iota\chi)}\right) \ = \ 
\i^{d_F/2} \Delta_F \, 
\frac{L(m, {}^{\varsigma\circ\iota}\chi)}{L(m+1, {}^{\varsigma\circ\iota}\chi)}.
\end{equation}
\end{con}

\bigskip

\subsubsection{\bf Thm.\,\ref{thm:main} $\implies$ Conj.\,\ref{con:deligne-hecke}}
This follows from the following: 

\begin{prop}
\label{prop:equality-signatures}
Let the notations and hypothesis be as in Thm.\,\ref{thm:main}. Then: 
$$
\frac{\varsigma(\i^{d_F/2} \Delta_F)}{\i^{d_F/2} \Delta_F}
\ = \ 
\frac{\varsigma(|\delta_{F/\Q}|^{1/2})}{|\delta_{F/\Q}|^{1/2})} \cdot 
\varepsilon_{{\bf n}, \iota}(\varsigma) \cdot 
\varepsilon_{\tilde{\bf n}, \iota}(\varsigma).
$$
\end{prop}

\begin{proof}
When $F$ is a CM-field, we have $\varepsilon_{{\bf n}, \iota}(\varsigma) \cdot 
\varepsilon_{\tilde{\bf n}, \iota}(\varsigma) = 1$ from \ref{sec:trivial-sign-CM}; the proof will follow from the following lemma: 

\begin{lemma}
\label{lem:compatibility_CM_field}
Let $F$ be a CM field. Then, as elements of $\C^\times/\Q^\times$, we have: 
$$
|\delta_{F/\Q}|^{1/2} \ = \ 
\i^{d_F/2} \cdot \Delta_F. 
$$
\end{lemma}

\begin{proof}[Proof of Lem.\,\ref{lem:compatibility_CM_field}]
Let $F_0$ be the maximal totally real subfield of $F;$ suppose $F = F_0(\sqrt{D})$, for a totally 
negative $D \in F_0;$ then $\Delta_F = \sqrt{N_{F_0/\Q}(D)}.$
For the absolute discriminant of $F$, using the well-known formula for the discriminant of a tower of fields, we have:
$$
\delta_{F/\Q} \ = \ \delta_{F_0/\Q}^{[F:F_0]} \cdot N_{F_0/\Q}(\delta_{F/F_0}) \ = \ 
\delta_{F_0/\Q}^2 \cdot N_{F_0/\Q}(\delta_{F/F_0}).
$$
Using the basis $\{1, \sqrt{D} \}$ we get the relative discriminant of $F/F_0$ as $\delta_{F/F_0} = 4D.$ Hence, taking square-root 
of the absolute value of $\delta_{F/\Q}$ we get: 
$$
\sqrt{|\delta_{F/\Q}|} \ = \ \sqrt{|N_{F_0/\Q}(D)|} \pmod{\Q^\times}.
$$
Since $D$ is totally negative, the sign of $N_{F_0/\Q}(D)$ is the same as $(-1)^{[F_0:\Q]} = (-1)^{d_F/2},$ i.e., 
$|N_{F_0/\Q}(D)| = (-1)^{d_F/2} N_{F_0/\Q}(D).$ Hence: 
$$
|\delta_{F/\Q}|^{1/2} \ = \ \sqrt{(-1)^{d_F/2}N_{F_0/\Q}(D)} 
 \ = \ \i^{d_F/2} \sqrt{N_{F_0/\Q}(D)} 
 \ = \ \i^{d_F/2} \Delta_F 
\pmod{\Q^\times}.
$$
\end{proof}

\noindent
Suppose now that $F$ is a totally imaginary field; the hypothesis of Thm.\,\ref{thm:main} guarantees that $F$ contains a CM subfield $F_1$ and its maximal 
totally real subfield $F_0.$ Lem.\,\ref{lem:compatibility_CM_field} generalizes to the following lemma: 

\medskip
\begin{lemma}
\label{lem:compatibility_tot_imag}
Let $F$ be a totally imaginary field, and $F_1$ its maximal CM subfield. 
Then, as elements of $\C^\times/\Q^\times$, we have: 
$$
|\delta_{F/\Q}|^{1/2} \ = \ 
\i^{d_F/2} \cdot \Delta_F \cdot \left(N_{F_1/\Q}(\delta_{F/F_1})\right)^{1/2}.
$$
\end{lemma}

\begin{proof}[Proof of Lem.\,\ref{lem:compatibility_tot_imag}]
Begin by noting that 
$\delta_{F/\Q} \ = \ \delta_{F_0/\Q}^{[F:F_0]} \cdot N_{F_0/\Q}(\delta_{F/F_0}),$ and since $[F:F_0] = 2[F:F_1]$ is even, we get 
$$
|\delta_{F/\Q}|^{1/2} \ = \ |N_{F_0/\Q}(\delta_{F/F_0})|^{1/2} \pmod{\Q^\times}.
$$
Next, since $\delta_{F/F_0} \ = \ \delta_{F_1/F_0}^{[F:F_1]} \cdot N_{F_1/F_0}(\delta_{F/F_1});$ using the $F_0$-basis $\{1, \sqrt{D}\}$ for 
$F_1$, one has $\delta_{F_1/F_0} = 4D;$ therefore, 
\begin{multline*}
N_{F_0/\Q}(\delta_{F/F_0}) \ = \ N_{F_0/\Q}(4D)^{[F:F_1]} \cdot N_{F_0/\Q}(N_{F_1/F_0}(\delta_{F/F_1})) \\ 
= \ N_{F_0/\Q}(D)^{[F:F_1]} \cdot N_{F_1/\Q}(\delta_{F/F_1}) \pmod{\Q^\times{}^2}.
\end{multline*}
Since $F_1/\Q$ is a CM-extension, $N_{F_1/\Q}(\delta_{F/F_1}) > 0$; hence, we get
$$
|\delta_{F/\Q}|^{1/2} \ = \  |N_{F_0/\Q}(\delta_{F/F_0})|^{1/2} \ = \ 
|N_{F_0/\Q}(D)|^{[F:F_1]/2} \cdot \left(N_{F_1/\Q}(\delta_{F/F_1})\right)^{1/2} \pmod{\Q^\times}. 
$$
Since $D \ll 0$ in $F_0$, we see that $(-1)^{[F_0:\Q]}N_{F_0/\Q}(D) > 0$. Hence, 
\begin{multline*}
|N_{F_0/\Q}(D)|^{[F:F_1]/2} \ = \ ((-1)^{[F_0:\Q]}N_{F_0/\Q}(D))^{[F:F_1]/2} \ = \ 
(\i^{[F_0:\Q]} \Delta_{F_1})^{[F:F_1]} \\ 
= \ \i^{[F_0:\Q][F:F_1]} \Delta_F \ = \ \i^{d_F/2} \Delta_F.
\end{multline*}
\end{proof}

\medskip
\begin{lemma}
\label{lem:equality-signs_tot_imag}
With notations as in Lem.\,\ref{lem:compatibility_tot_imag} and Thm.\,\ref{thm:main}, we have the equality of signatures:
$$
\varepsilon_{{\bf n}, \iota}(\varsigma) \cdot 
\varepsilon_{\tilde{\bf n}, \iota}(\varsigma) \ = \ 
\frac{\varsigma\left(N_{F_1/\Q}(\delta_{F/F_1})^{1/2}\right)}
{N_{F_1/\Q}(\delta_{F/F_1})^{1/2}}.
$$
\end{lemma}

\begin{proof}[Proof of Lem.\,\ref{lem:equality-signs_tot_imag}]
For the left hand side, 
recall that $\varepsilon_{{\bf n}, \iota}(\varsigma)$ (resp., $\varepsilon_{\tilde{\bf n}, \iota}(\varsigma)$) 
is the sign of the permutation $\pi_{{\bf n}, \iota}(\varsigma)$ (resp., $\pi_{\tilde{\bf n}, \iota}(\varsigma)$)
induced by the map 
$\varsigma \circ - : \Phi_{{\bf n}, \iota} \to \Phi_{{\bf n}, \varsigma \circ \iota}$ 
(resp., $\Phi_{\tilde{\bf n}, \iota} \to \Phi_{\tilde{\bf n}, \varsigma \circ \iota}$). For brevity, let $\Phi = \Phi_{{\bf n}, \iota}$, then 
$\bar\Phi = \Phi_{\tilde{\bf n}, \iota}$. Recall from Prop.\,\ref{prop:descend-to-F_1} that ${\bf n}$ is the base-change of ${\bf m}$, hence
the CM-type $\Psi = \Phi_{{\bf m}, \iota}$ for the CM-field $F_1$ has the property that $\Phi \to \Psi$ under the restriction map; also 
$\bar\Phi \to \bar\Psi.$ One has the following relation with respect to conjugating by a Galois element $\varsigma$: 
$$
\xymatrix{
\Phi \cup \bar\Phi \ar[rr]^{\varsigma \circ -} \ar[d]& & \varsigma(\Phi) \cup \varsigma(\bar\Phi) \ar[d] \\
\Psi \cup \bar\Psi \ar[rr]^{\varsigma \circ -} & & \varsigma(\Psi) \cup \varsigma(\bar\Psi)
}
$$
Fixing an ordering on $S_\infty(F_1)$, the bottom row involves two permutations corresponding to 
$\Psi \to \varsigma(\Psi)$ and  $\bar\Psi \to  \varsigma(\bar\Psi)$, but from 
Lem.\,\ref{lem:CM-field-embeddings} it follows that they are equal, hence $\varepsilon_{{\bf m}, \iota}(\varsigma) = \varepsilon_{\tilde{\bf m}, \iota}(\varsigma).$
However, the two permutations and their signatures corresponding to $\Phi \to \varsigma(\Phi)$ and  $\bar\Phi \to \varsigma(\bar\Phi)$ need not be equal.

\smallskip
For the right hand side: let $\Sigma_{F_1} = \{\theta_1,\dots, \theta_{d_1}\};$ $d_1 = [F_1:\Q];$ then for any $x \in F_1^\times$ one has
$N_{F_1/\Q} = \prod_{j=1}^{d_1} \theta_j(x) > 0.$ Also, suppose $\{\rho_1,\dots,\rho_k\}$ denotes all embeddings of $F$ into $\bar{F_1}$ over $F_1,$
and $\{\omega_1,\dots,\omega_k\}$ is an $F_1$-basis for $F$, then $\delta_{F/F_1} = \det[\rho_i(\omega_j)]^2.$ Putting together, we get
$$
N_{F_1/\Q}(\delta_{F/F_1}) \ = \ \prod_{\theta \in \Sigma_{F_1}} \theta(\det[\rho_i(\omega_j)]^2) \ = \ 
\prod_{\theta \in \Sigma_{F_1}} \det[\rho_i^\theta(\omega_j)]^2,
$$
where $\{\rho_1^\theta, \dots, \rho_k^\theta\}$ is the set of all embeddings of $F$ into $\bar{\Q}$ that restrict to $\theta : F_1 \to \bar{\Q}.$ Hence, 
\begin{equation}
\label{eqn:norm-matrix}
N_{F_1/\Q}(\delta_{F/F_1})^{1/2} \ = \ 
\pm \det
\left[\begin{array}{cccc}
[\rho_i^{\theta_1}(\omega_j)] & & & \\
 & [\rho_i^{\theta_2}(\omega_j)] & & \\
 & & \ddots & \\
 & & & [\rho_i^{\theta_{d_1}}(\omega_j)]
\end{array}\right],
\end{equation}
where the appropriate sign $\pm$ is chosen to make the right hand side positive. Denote by $A^{\theta}$ the $k \times k$ block $[\rho_i^{\theta}(\omega_j)]$. 
To the above equation apply a Galois element $\varsigma$. To see how the determinant changes, note that 
then the blocks along 
the diagonal are permuted according to how $\varsigma$ permutes $\Sigma_{F_1}$ which is the same as $\Psi \to \varsigma(\Psi)$ and  $\bar\Psi \to  \varsigma(\bar\Psi),$ 
and since $\varepsilon_{{\bf m}, \iota}(\varsigma) = \varepsilon_{\tilde{\bf m}, \iota}(\varsigma).$ Just a permutation of the blocks of a block diagonal matrix does not change the determinant. 
Next, if $\theta \mapsto \varsigma\theta$ then the rows of the block 
$A^{\theta}$ are permuted to get the rows of $A^{\varsigma\theta}.$ Fix an ordering on $S_\infty(F)$ and correspondingly suppose $\Phi = \{\theta_1,\dots, \theta_{r_1}\}$,  
where $r_1 = d_1/2.$ The blocks of the block-diagonal matrix ${\rm diag}(A^{\theta_1},\dots, A^{\theta_{r_1}})$ (which is `half' the matrix in \eqref{eqn:norm-matrix}) 
get permuted to give ${\rm diag}(A^{\varsigma\theta_1},\dots, A^{\varsigma\theta_{r_1}})$, and within the blocks the net effect of the permutation is exactly the permutation 
$\pi_{{\bf n}, \iota}(\varsigma)$ from Sec.\,\ref{sec:two-signatures}, and so the determinant changes by the signature 
$\varepsilon_{{\bf n}, \iota}(\varsigma)$. Similarly, working with the `other-half' ${\rm diag}(A^{\bar\theta_1},\dots, A^{\bar\theta_{r_1}})$ we get the signature 
$\varepsilon_{\tilde{\bf n}, \iota}(\varsigma).$ Hence, the effect of applying $\varsigma$ to the matrix in \eqref{eqn:norm-matrix} changes its determinant by 
$\varepsilon_{{\bf n}, \iota}(\varsigma)\varepsilon_{\tilde{\bf n}, \iota}(\varsigma).$ 
\end{proof}

\smallskip
Lem.\,\ref{lem:compatibility_tot_imag} and 
Lem.\,\ref{lem:equality-signs_tot_imag} prove Prop.\,\ref{prop:equality-signatures}. 
\end{proof}

\bigskip
\subsubsection{\bf Some final remarks}

Using Prop.\,\ref{prop:equality-signatures} the result in Thm.\,\ref{thm:main} 
may be restated as the following theorem: 

\begin{thm}
\label{thm:main-restated} 
Let the notations and hypothesis be as in Thm.\,\ref{thm:main}. 
Then: 
$$
\i^{d_F/2} \Delta_F \, \frac{L(m, {}^\iota\chi)}{L(m+1, {}^\iota\chi)} \ \in \iota(E), 
$$
and, furthermore, for every $\varsigma \in \Gal(\bar\Q/\Q)$ we have the reciprocity law: 
$$
\varsigma \left(\i^{d_F/2} \Delta_F \, \frac{L(m, {}^\iota\chi)}{L(m+1, {}^\iota\chi)}\right)
\ \ = \ \ 
\i^{d_F/2} \Delta_F \, 
\frac{L(m, {}^{\varsigma\circ\iota}\chi)}{L(m+1, {}^{\varsigma\circ\iota}\chi)}.
$$
\end{thm} 

\medskip
We conclude this article by drawing attention to following piquant detail in the special case when $F$ is a CM field. As before, $F_0$ is its maximal totally real subfield. Suppose $\omega_{F/F_0}$ is the quadratic Hecke character of $F_0$ associated to $F/F_0$ by class field theory, and let 
$\G(\omega_{F/F_0})$ denote the associated quadratic Gau\ss\ sum. Then 
\begin{equation}
\label{eqn:piquant} 
|\delta_{F/\Q}|^{1/2} \ = \ \i^{d_F/2} \, \Delta_F \ = \ \i^{d_F/2} \, \G(\omega_{F/F_0})  \pmod{\Q^\times}; 
\end{equation}
the first congruence is exactly Lem.\,\ref{lem:compatibility_CM_field}, and the second congruence is well-known (see Prop.\,33 of \cite{raghuram-CM}). Such a Gau\ss\ sum naturally appears if we were to use automorphic induction and study the $L$-function of a Hecke character of the CM field $F$ as the $L$-function of a Hilbert modular form over $F_0$; see  
\cite[Sect.\,5.5]{raghuram-tot-imag}. 
More generally, as shown in \cite{raghuram-CM}, the Gau\ss\ sum naturally appears when we study the standard $L$-functions 
for $\GL(n)/F$ via automorphic induction as the standard $L$-functions of $\GL(2n)/F_0$.  

\medskip

\begin{small}
\bibliographystyle{plain}

\end{small}

\bigskip

\end{document}